\documentclass{gtpart}     
\gtart  
\usepackage{amsfonts,verbatim,lscape,color,tensor,url}
\usepackage[enableskew]{youngtab}
\usepackage[all]{xy}
\usepackage{todonotes}

\usepackage[nameinlink,capitalise,noabbrev]{cleveref}
\title{The character of the total power operation}

\author{Tobias Barthel}
\givenname{Tobias}
\surname{Barthel}
\address{Max Planck Institute for Mathematics, Bonn, Germany}
\email{tbarthel@mpim-bonn.mpg.de}
\urladdr{http://people.mpim-bonn.mpg.de/tbarthel/}

\author{Nathaniel Stapleton}
\givenname{Nathaniel}
\surname{Stapleton}
\address{Max Planck Institute for Mathematics, Bonn, Germany}
\email{nstapleton@mpim-bonn.mpg.de}
\urladdr{http://guests.mpim-bonn.mpg.de/nstapleton/}

\arxivreference{1502.01987}
\arxivpassword{kerms}

\volumenumber{}
\issuenumber{}
\publicationyear{}
\papernumber{}
\startpage{}
\endpage{}
\doi{}
\MR{}
\Zbl{}
\received{}
\revised{}
\accepted{}
\published{}
\publishedonline{}
\proposed{}
\seconded{}
\corresponding{}
\editor{}
\version{}

\newtheorem{thm}[subsection]{Theorem}
\newtheorem{prop}[subsection]{Proposition}
\newtheorem{cor}[subsection]{Corollary}
\newtheorem{lemma}[subsection]{Lemma}

\newtheorem*{thmA}{Theorem A}
\newtheorem*{thmB}{Theorem B}
\newtheorem*{thmC}{Theorem C}

\theoremstyle{definition}
\newtheorem{definition}[subsection]{Definition}

\newtheorem{example}[subsection]{Example}

\newtheorem{remark}[subsection]{Remark}

\newcommand{\powser}[1]{[\![#1]\!]}
\newcommand{\pdiv}{$p$-divisible }

\newcommand{\G}{\mathbb{G}}
\newcommand{\F}{\mathbb{F}}

\renewcommand{\Q}{\mathbb{Q}}
\newcommand{\Po}{\mathbb{P}}

\renewcommand{\Z}{\mathbb{Z}}
\newcommand{\N}{\mathbb{N}}
\newcommand{\QZ}[1]{(\Q_p/\Z_p)^{#1}}

\renewcommand{\R}{\mathbb{R}}

\newcommand{\al}{\alpha}
\newcommand{\Lk}{\Lambda_k}

\newcommand{\lra}[1]{\overset{#1}{\longrightarrow}}

\newcommand{\Prod}[1]{\underset{#1}{\prod}}
\newcommand{\Oplus}[1]{\underset{#1}{\bigoplus}}

\newcommand{\Coprod}[1]{\underset{#1}{\coprod}}
\newcommand{\Colim}[1]{\underset{#1}{\colim}}

\newcommand{\E}{E_{n}}

\renewcommand{\D}{D_{\infty}}

\newcommand{\trans}{\text{trans}}
\newcommand{\pk}{\underline{p^k}}
\newcommand{\uG}{\underline{G}}
\newcommand{\um}{\underline{m}}
\newcommand{\mund}{\underline{m}}

\DeclareMathOperator{\Aut}{Aut}

\DeclareMathOperator{\im}{im}

\DeclareMathOperator{\Hom}{Hom}
\DeclareMathOperator{\colim}{colim}

\DeclareMathOperator{\Iso}{Iso}
\DeclareMathOperator{\End}{End}
\DeclareMathOperator{\Isog}{Isog}

\DeclareMathOperator{\Spf}{Spf}

\DeclareMathOperator{\Id}{Id}

\DeclareMathOperator{\Res}{Res}
\DeclareMathOperator{\Sub}{Sub}

\DeclareMathOperator{\Level}{Level}

\DeclareMathOperator{\Def}{Def}
\DeclareMathOperator{\Sum}{Sum}
\DeclareMathOperator{\Tr}{Tr}
\DeclareMathOperator{\GL}{GL}

\DeclareMathOperator{\lat}{\mathbb{L}}
\DeclareMathOperator{\qz}{\mathbb{T}}

\DeclareMathOperator{\id}{id}
\newcommand{\upperRomannumeral}[1]{\uppercase\expandafter{\romannumeral#1}}

\newcommand{\Fam}[1]{\mathcal{F}({#1})}

\begin{document}

\begin{abstract}    
In this paper we compute the total power operation for the Morava $E$-theory of any finite group up to torsion. Our formula is stated in terms of the $\GL_n(\Q_p)$-action on the Drinfeld ring of full level structures on the formal group associated to $E$-theory. It can be specialized to give explicit descriptions of many classical operations. Moreover, we show that the character map of Hopkins, Kuhn, and Ravenel from $E$-theory to $\GL_n(\Z_p)$-invariant generalized class functions is a natural transformation of global power functors on finite groups.  
\end{abstract}
\maketitle


\section{Introduction}

Power operations and their variants are ubiquitous throughout homotopy theory. The Steenrod operations on mod $p$ cohomology and the Adams operations on topological $K$-theory are familiar examples. These operations have proven extremely useful; for instance, the Adams operations were used to give a short and elegant proof of the Hopf invariant one problem~\cite{adamsatiyahhopf}. More generally, many cohomology theories are equipped with this extra structure which is a consequence of an $E_{\infty}$-ring structure on the representing spectrum.


It is a theorem of Goerss, Hopkins, and Miller~\cite{structuredmoravae} that Morava $E$-theory admits a unique $E_{\infty}$-ring structure. In the homotopy category, the $E_{\infty}$-ring structure manifests itself as an $H_{\infty}$-ring structure which is equivalent to the data of a collection of multiplicative cohomology operations known as total power operations. The study of these operations began in earnest in \cite{isogenies} in which a connection is established between the total power operations and isogenies of the formal group associated to $\E$. In this paper, we give a formula for the total power operations applied to a finite group in terms of the action of $\GL_n(\Q_p)$ on the Drinfeld ring of full level structures on the formal group associated to $\E$. This is the same action that appears in the local Langlands correspondence~\cite{carayolltt}.


\subsection{Motivation and background}

Fix a prime $p$, a natural number $n$, and a height $n$ formal group law $\Gamma$ over $\kappa$, a perfect field of characteristic $p$. This data determines the $E_{\infty}$-ring spectrum $\E$ known as Morava $E$-theory. It has been studied extensively because of its close relationship to several areas of mathematics: Work of Devinatz and Hopkins~\cite{devinatzhopkinsfixedpoints} as well as Rognes~\cite{rognesgalois} demonstrates that $\E$ is a Galois extension of the $K(n)$-local sphere, a fundamental object of the stable homotopy category. It is closely connected to algebraic geometry as the coefficient ring $\E^0$ carries the universal deformation $\G$ of $\Gamma$. Morava $E$-theory is also related to representation theory, being a generalization of $p$-adic $K$-theory that admits a well-behaved character theory. In this paper we will take advantage of the last two relationships to study the total power operations determined by the unique $E_{\infty}$-ring structure on $\E$.

Let $X$ be a topological space and let $E\Sigma_m \times_{\Sigma_m} X^m$ be the Borel construction for the canonical $\Sigma_m$-action on $X^m$. The total power operations are natural multiplicative non-additive maps
\[
\Po_m \colon \E^0(X) \lra{} \E^0(E\Sigma_{m} \times_{\Sigma_{m}} X^{^m}),
\]
defined for all $m>0$. These maps are quite mysterious and notoriously difficult to compute, see~\cite{rezkpowerops}. At height $2$, $\Po_m$ has been explicitly determined for $X$ a point when $p=2$ and $m = 2$ and when $p=3$ and $m=3$, see~\cite{rezkpowercalc, zhupo}. Above height $2$ there have been no explicit computations. Many of the most important operations on $\E$, including the Adams operations, Hecke operations, and the logarithm~\cite{logarithmic}, can be built out of the total power operations using various simplifications of $\Po_m$. 

A useful simplification of $\Po_m$ is obtained as the restriction of the total power operation along the diagonal $X \lra{} X^{m}$. This produces a map
\[
P_m\colon\E^0(X) \lra{} \E^0(B\Sigma_{m} \times X) \cong \E^0(B\Sigma_{m}) \otimes_{\E^0} \E^0(X),
\]
the isomorphism being a consequence of the freeness of $\E^0(B\Sigma_{m})$ as a module over $\E^0$ \cite[Proposition 3.6]{etheorysym}. Let $I \subset \E^0(B\Sigma_{p^k})$ be the image of the transfer along the inclusion $\Sigma_{p^{k-1}}^p \subset \Sigma_{p^k}$, then the quotient
\[
P_{p^k}/I\colon\E^0(X) \lra{} \E^0(B\Sigma_{p^k})/I \otimes_{\E^0} \E^0(X)
\]
is a ring map \cite[Chapter VIII, Proposition 1.4.(iv)]{bmms} \cite[Lemma 8.11]{etheorysym}. Thus it is reasonable to hope that $P_{p^k}/I$ may be attacked using algebraic geometry associated to $\G$. For $X = BA$, where $A$ is a finite abelian group, this was accomplished by Ando~\cite{isogenies} and Ando, Hopkins, and Strickland~\cite{ahs}. In constrast to these approaches, we apply a form of character theory available for $\E$ in order to simplify $\Po_m$.


Generalized character theory for Morava $E$-theory was constructed by Hopkins, Kuhn, and Ravenel in \cite{hkr} building on work of Adams on $p$-adic $K$-theory in \cite{Adams-Maps2} and extended by the second author in \cite{tgcm}. An introduction to the subject is available in~\cite{hkrsintroduction}. As $E$-theory is constructed homotopy theoretically, it is surprising that $\E^0(BG)$ behaves so much like a (completed) ``higher" representation ring. 


Let $\lat = \Z_{p}^n$ and $\qz = \lat^*$, the $p$-typical Pontryagin dual, so that there is a non-canonical isomorphism $\qz \cong \QZ{n}$. Hopkins, Kuhn, and Ravenel construct a $p^{-1}\E^0$-algebra $C_0$ that corepresents isomorphisms of \pdiv groups between $\qz$ and $\G$:
\[
\hom(C_0, R) \cong \Iso(R \otimes \qz, R \otimes \G).
\]
Thus there is an obvious action of $\Aut(\qz)$, the automorphisms of $\qz$, on $C_0$. 

Let $Cl_n(G,C_0)$ be the ring of generalized class functions on $G$ taking values in $C_0$. Concretely, $Cl_n(G,C_0)$ is the set of $C_0$-valued functions on $\hom(\lat, G)_{/\sim}$, the quotient of the set $\hom(\lat,G)$ by the conjugation action of $G$. Note that $\Aut(\qz)$ is contravariantly isomorphic to $GL_n(\Z_p)$ by taking the Pontryagin dual. There is an $\Aut(\qz)$-action on $Cl_n(G,C_0)$ given by combining the action of $\Aut(\qz)$ on $\hom(\lat,G)_{/\sim}$ by precomposition with the Pontryagin dual and the action of $\Aut(\qz)$ on $C_0$ \cite[Section 6.3]{hkr}.

In \cite[Section 6]{hkr}, Hopkins, Kuhn, and Ravenel construct a map of $\E^0$-algebras called the generalized character map
\[
\chi\colon\E^0(BG) \lra{} Cl_n(G,C_0).
\]
Theorem C in \cite{hkr} proves that the induced map
\[
C_0 \otimes \chi \colon C_0 \otimes_{\E^0} \E^0(BG) \lra{\cong} Cl_n(G,C_0)
\]
is an isomorphism for any finite group. Theorem C further states that the isomorphism $C_0 \otimes \chi$ is $\Aut(\qz)$-equivariant and restricts to an isomorphism on fixed points
\[
p^{-1}\E^0(BG) \lra{\cong} Cl_n(G, C_0)^{\Aut(\qz)}.
\]

This brings us to the motivating question for this paper. Does there exist a multiplicative natural transformation on generalized class functions that is compatible with the total power operation for $\E$ through the generalized character maps of \cite{hkr}:
\[
\xymatrix{\E^0(BG) \ar[r]^-{\Po_m} \ar[d]_{\chi} & \E^0(BG\wr \Sigma_{m}) \ar[d]^{\chi} \\ Cl_n(G,C_0) \ar@{-->}[r]^-{\exists ?} & Cl_n(G \wr \Sigma_{m},C_0).}
\]
There is no formal reason why this might be possible,\footnote{It is not possible to base change $\Po_m$ to $C_0$ because $\Po_m$ is not a ring map. Further, simplifications of $\Po_m$ such as $P_{p^k}/I$, which are ring maps, are not $\E^0$-algebra maps.} and yet the construction of such a natural transformation is one of the goals of this paper. 

\subsection{Main results} In fact, we construct an infinite family of multiplicative natural transformations that answer the question. In order to state the results precisely we need to establish some notation. 

Let $\Isog(\qz)$ be the monoid of endoisogenies of $\qz$, i.e., the monoid of endomorphisms with finite kernel, and let $\Sub(\qz)$ be the set of finite subgroups in $\qz$. There is an $\Aut(\qz)$-principal bundle
\[
\Isog(\qz) \twoheadrightarrow \Sub(\qz)
\]
given by taking an isogeny to its kernel. For each section $\phi$ of this principal bundle we construct a multiplicative natural transformation
\[
\Po^{\phi}_m \colon Cl_n(-,C_0) \Rightarrow Cl_n(-\wr \Sigma_m,C_0)
\]
that is compatible with the total power operation $\Po_m$ through the character map $\chi$. We refer the reader to the paragraphs leading up to \Cref{def:rationalpo} for the explicit formula for $\Po^{\phi}_m$. 

\begin{thmA}[\Cref{clmainthm}]
For all $n,m \geq 0$, let $\Po_m$  be the total power operation for Morava $\E$, let $\phi$ be a section of the principal bundle above, and let $\chi$ be the generalized character map. There is a commutative diagram
\[
\xymatrix{\E^0(BG) \ar[r]^-{\Po_m} \ar[d]_{\chi} & \E^0(BG \wr \Sigma_{m}) \ar[d]^{\chi} \\ Cl_n(G,C_0) \ar[r]^-{\Po^{\phi}_{m}} & Cl_n(G\wr \Sigma_{m},C_0)}
\]
natural in $G$.
\end{thmA}

As part of \cref{rationalized}, we further prove that each of these ``total power operations" $\Po^{\phi}_m$ may be restricted to the $\Aut(\qz)$-fixed points. In \cref{independent}, we show that the result is independent of $\phi$ and we call it the ``rational total power operation"
\[
\Po^{\Q}_{m}\colon p^{-1}\E^0(BG) \lra{} p^{-1}\E^0(BG \wr \Sigma_m).
\]

\begin{thmB}[\Cref{rationalized}]
With the notation of Theorem A, there is a commutative diagram
\[\xymatrix{\E^0(BG) \ar[r]^-{\mathbb{P}_m} \ar[d]_{\chi} & \E^0(BG \wr \Sigma_{m}) \ar[d]^{\chi} \\
p^{-1}\E^0(BG) \ar[r]_-{\mathbb{P}^{\Q}_{m}} & p^{-1}\E^0(BG \wr \Sigma_{m}).}\]
\end{thmB}

We note that the rationalization $p^{-1}\E^0(BG)$ retains much of the information contained in $\E^0(BG)$, so our result approximates the total power operation closely. In fact,  $\E^0(BG)$ is finitely generated and free for many finite groups $G$. These are the so-called good groups\cite{hkr,bscentralizers,schuster32}. Finally, we show that $\mathbb{P}^{\Q}_{m}$ deserves to be called the rational total power operation by proving it is a global power functor in the sense of \cite[Section 4]{globalgreen}.

\begin{thmC}[\Cref{ratpoglobalpower}, \Cref{charglobalpower}]
The rational total power operation $\Po^{\Q}_{m}$ is a global power functor in such a way that the character map 
\[
\chi\colon \E^0(B-) \lra{} p^{-1}\E^0(B-)
\]
is a map of global power functors. 
\end{thmC}

Note that Theorem C is not a direct consequence of Theorem B because Theorem B does not say anything about composites of rational total power operations.

\subsection{Outline}

In \Cref{sec:notation}, we establish notation and terminology that will be used in the rest of the paper.

The first part of this paper, comprised of \Cref{sec:grouptheory} through~\Cref{sec:rationalpo}, contains results that do not rely on Morava $E$-theory. We hope that some of the ideas in these sections might be of interest to readers outside of stable homotopy theory. The goal of \cref{sec:grouptheory} is a thorough study of conjugacy classes of commuting elements in wreath products $G \wr \Sigma_m$ and their algebro-geometric interpretation. We suspect that this material is well-known, but were unable to locate a reference in the literature. \Cref{sec:isogenies} deals with the action of the isogenies of $\qz$ on the Drinfeld ring of full level-structures on $\G$. This is fundamental for the rest of the paper and extends work of~\cite{hkr}. For all $m \geq 0$ we construct multiplicative natural transformations $\Po^{\phi}_{m}$ on class functions in \Cref{sec:rationaltrafo} and study their basic properties. The question of whether one can choose a section $\phi$ such that the associated transformation $\Po^{\phi}_{m}$ inherits the structure of a global power functor is considered in \Cref{sec:rationalpo}. We give an affirmative answer for heights $1$ and $2$ by stipulating an explicit solution. 

The second part begins in \Cref{sec:etheoryrecolletions} with some recollections on Morava $E$-theory and power operations. As a first step towards the proof of our main theorem, we consider the case of abelian groups in \Cref{sec:proofabelian}. Using work of Schlank and Stapleton \cite{genstrickland}, we extend the algebro-geometric description of the additive total power operations for abelian groups due to Ando, Hopkins, and Strickland~\cite{ahs}. Inspired by Artin induction, \Cref{sec:proofgeneral} then proves Theorem A by reducing it to the case of abelian groups. Our work culminates in the final \Cref{sec:maintheorem}, where the previous results are combined to descend the multiplicative natural transformations to rationalized $E$-theory. We prove uniqueness and establish the global power structure on the resulting multiplicative natural transformation. 

\subsection*{Acknowledgements}
It is a pleasure to thank Charles Rezk and Tomer Schlank for their valuable input. Without some of their ideas this project would not have gotten off the ground. We would like to thank Nora Ganter, Justin Noel, and Chris Schommer-Pries for helpful conversations, and Drew Heard for useful comments on an earlier draft. The referee and Haynes Miller made many comments leading to an improvement of the exposition of the paper. We also thank the Max Planck Institute for Mathematics for its hospitality.

\section{Notation and conventions}\label{sec:notation}
Let $G$ be a finite group and fix a prime $p$ and a natural number $n \geq 0$. We set $\lat = \Z_p^n$ and let $\hom(\lat, G)$ be the set of continuous homomorphisms of groups from $\lat$ to $G$. Fixing a basis of $\lat$ gives a bijection
\[
\hom(\lat, G) \cong \{(g_1, \ldots, g_n)| [g_i, g_j] = e \text{ and } g_{i}^{p^k} = e \text{ for some } k > 0\},
\]
where the target is the set of $n$-tuples of pairwise commuting prime-power order elements in $G$. The group $G$ acts on this set by conjugation; we write $\hom(\lat, G)_{/\sim}$ for the set of conjugacy classes. Moreover, we will write $[\al]$ for the conjugacy class of a map $\al\colon \lat \lra{} G$.

Let $G \wr \Sigma_{m}$ be the wreath product of $G$ with $\Sigma_{m}$, which is constructed as the semidirect product $G^{m} \rtimes \Sigma_{m}$ with respect to the permutation action of $\Sigma_m$ on $G^{m}$. The wreath product fits into a short exact sequence
\[
1 \lra{} G^{m} \lra{} G \wr \Sigma_{m} \lra{\pi} \Sigma_{m} \lra{} 1.
\]
We will always write $\pi$ for the projection onto the symmetric group, and represent elements in $G \wr \Sigma_{m}$ as $(g_1,\ldots,g_m;\sigma)$ with $g_i \in G$ and $\sigma \in \Sigma_m$. We may abbreviate this to $(\bar{g};\sigma)$, where $\bar{g} = (g_1, \ldots, g_m)$ and $\bar{g}_i = g_i$. Multiplication in the wreath product is given by
\begin{equation} \label{wreathformula}
(\bar{g};\sigma)(\bar{h};\tau) = (\bar{g} (\sigma \bar{h}); \sigma \tau),
\end{equation}
where $(\sigma \bar{h})_i = h_{\sigma^{-1}i}$.

We call a map $\al\colon\lat \lra{} \Sigma_{m}$ transitive if the map represents a transitive $\lat$-set of order $m$. Equivalently, the image $\im \al$ of the map $\al$ is a transitive abelian subgroup of $\Sigma_{m}$. This can only occur when $m = p^k$. We say that a map 
\[
\al \colon \lat \lra{} G \wr \Sigma_{p^k}
\]
is transitive if $\pi \al$ is transitive. Let $\hom(\lat, G\wr \Sigma_{p^k})^{\trans}$ be the set of transitive maps and let $\hom(\lat, G\wr \Sigma_{p^k})_{/\sim}^{\trans}$ be the set of conjugacy classes of transitive maps.

Pontryagin duality is pervasive throughout this paper. We will use the notation $(-)^*$ for the ($p$-typical) Pontryagin duality endofunctor on abelian groups. Thus for an abelian group $A$, we write
\[
A^* = \hom(A, \Q_p/\Z_p)
\] 
for the dual abelian group. For a map of abelian groups $A \lra{f} B$, we write $f^*$ for the dual map. Let $\qz = \lat^*$, then we have a non-canonical isomorphism
\[
\qz \cong \QZ{n}.
\]
Let $[p^k] \colon \qz \rightarrow \qz$ the multiplication by $p^k$ map and let $\qz[p^k]$ be the $p^k$-torsion in $\qz$.

Let $\Sub_{p^k}(\qz)$ be the set of subgroups of order $p^k$ in $\qz$. Let $\Sum_{m}(\qz)$ be the set of formal sums of subgroups $\bigoplus_i H_i$ with $H_i \subset \qz$ and $\sum_i |H_i| = m$. It is clear that $\Sub_{p^k}(\qz)$ is a subset of $\Sum_{p^k}(\qz)$. Given a subgroup $H \subset \qz$ of order $p^k$, let $f_H\colon H \hookrightarrow \qz$ be the inclusion. The Pontryagin dual of $f_H$,
\[
f_{H}^* \colon \lat \twoheadrightarrow H^*,
\]
has the associated short exact sequence
\[
0 \rightarrow \ker(f_{H}^*) \hookrightarrow \lat \twoheadrightarrow H^* \rightarrow 0.
\]
The kernel of $f_{H}^*$ is canonically isomorphic to $(\qz/H)^*$ and non-canonically isomorphic to $\lat$. We set $\lat_H = \ker(f_{H}^*) \subset \lat$.

More generally, we define $\Sub_{p^k}(\qz,G)$ to be the set of pairs consisting of a subgroup of order $p^k$, $H \subset \qz$, and a conjugacy class in $\hom(\lat_H, G)$. We will write such a pair as $(H, [\al])$. We define $\Sum_{m}(\qz, G)$ to be the collection of formal sums of pairs $\bigoplus_i (H_i, [\al_i])$, where $H_i \subset \qz$ such that $\sum_i |H_i| = m$, and 
\[
[\al_i] \in \hom(\lat_{H_i}, G)_{/\sim}.
\]
If $e$ is the trivial group, then we have 
\[
\Sub_{p^k}(\qz,e) = \Sub_{p^k}(\qz)
\]
and
\[
\Sum_{m}(\qz,e) = \Sum_{m}(\qz).
\]
Finally, given an isogeny $\phi_H\colon \qz \lra{} \qz$ with kernel $H$, it is necessary to have notation for the induced triangle
\[
\xymatrix{& \qz \ar@{->>}[rd]^{\phi_H} \ar@{->>}[dl]_{q_H} &  \\ 
\qz/H \ar[rr]_{\psi_H}^{\cong} && \qz.}
\]
The Pontryagin dual triangle then takes the form
\[
\xymatrix{& \lat & \\
 \lat_H \ar@{^{(}->}[ru]^{q_{H}^*} & &  \lat. \ar@{_{(}->}[ul]_{\phi_{H}^*} \ar[ll]^{\psi_{H}^*}_{\cong} }
\]

\section{Conjugacy classes in wreath products}\label{sec:grouptheory}

The purpose of this section is to establish a canonical bijection between conjugacy classes in wreath products and formal sums of subgroups in $\qz$:
\[
\xymatrix{\hom(\lat, G \wr \Sigma_{m})_{/\sim} \ar[r]^-{F}_-{\cong} & \Sum_{m}(\qz,G).}
\]
In fact, we have canonical bijections, which are natural in $G$, as indicated in the table below.

\begin{center}
\begin{tabular}{|l|l|}
\hline
Conjugacy classes & Subgroups \\
\hline
$\hom(\lat,\Sigma_{p^k})^{\trans}_{/\sim}$ & $\Sub_{p^k}(\qz)$\\
$\hom(\lat, \Sigma_{m})_{/\sim}$ & $\Sum_{m}(\qz)$ \\
$\hom(\lat, G \wr \Sigma_{p^k})^{\trans}_{/\sim}$ & $\Sub_{p^k}(\qz,G)$ \\
$\hom(\lat, G \wr \Sigma_{m})_{/\sim}$ & $\Sum_{m}(\qz,G)$ \\
\hline
\end{tabular}
\end{center}

In the case that $\lat = \Z_p$, these isomorphisms can be obtained from Section 7.7 of \cite{zel}. 


We will use $F$ for any of the bijections in the table (they are all special cases of the bottom entry of the table). The top two bijections are well-known; we will begin by describing $F$ in these cases. 

The canonical bijection between $\hom(\lat,\Sigma_{p^k})^{\trans}_{/\sim}$ and $\Sub_{p^k}(\qz)$ is constructed as follows: Let $\lat \lra{\al} \Sigma_{p^k}$ be a transitive map and let $\lat ' = \ker \al$. The map sends $[\al]$ to the kernel of the Pontryagin dual of the map $\lat ' \hookrightarrow \lat$
\[
H = \ker(\lat^* \twoheadrightarrow (\lat ')^*).
\]
Since a conjugate map has the same kernel, this assignment does not depend on the chosen representative $\al$ for the conjugacy class $[\al]$. The resulting subgroup $H$ is an invariant of the conjugacy class $[\al]$. There is another description of this subgroup. The map $\al$ determines a transitive $\lat$-set of order $p^k$. The lattice $\lat '$ is the stabilizer of any element in the set and then $H$ may be formed as the kernel above.

The second bijection is constructed similarly. A map $\al \colon \lat \rightarrow \Sigma_m$ is an $\lat$-set $X$ of order $m$. Let $X = \Coprod{i} X_i$ be the decomposition of $X$ into transitive $\lat$-sets. By following the recipe above, each component gives a subgroup $H_i \subset \qz$. Since a conjugate map corresponds to an isomorphic $\lat$-set, this collection of subgroups of $\qz$ is an invariant of the conjugacy class:  
\[
[\al] \mapsto F([\al]) = \bigoplus_i H_i.
\]

Now we establish the third and fourth bijections in the table above. Let $\uG$ be a left $G$-torsor (a free transitive left $G$-set) and let $\um$ be a set with $m \in \N$ elements and a trivial $G$-action. We will make use of three basic lemmas regarding $G$-sets of the form $\uG \times \um$. 

Let $\bar{x} = (x_1, \ldots, x_m)$ be a set of generators for $\uG \times \um$ as a $G$-set. This choice of generators induces an isomorphism
\begin{equation} \label{wreathiso}
\bar{x}^* \colon \Aut_G(\uG \times \um) \cong G \wr \Sigma_m.
\end{equation}
We may be more explicit about this isomorphism. An automorphism $f \in \Aut_G(\uG \times \um)$ permutes the $m$ $G$-torsors, call this permutation $\sigma_f$. Thus $f(x_i) = g_{\sigma_f(i)}x_{\sigma_f(i)}$ for some $g_{\sigma_{f}(i)} \in G$. The isomorphism is given by
\[
f \mapsto \bar{x}^*(f) = (g_{1}^{-1}, \ldots, g_{m}^{-1}; \sigma_f) = (\bar{g}^{-1};\sigma_f).
\]
This defines a homomorphism: If $h \in \Aut_G(\uG \times \um)$ maps to
\[
h \mapsto (\bar{k}^{-1}; \sigma_h) \in G \wr \Sigma_m,
\]
then 
\[
(h\circ f)(x_i) = h(g_{\sigma_f(i)}x_{\sigma_f(i)}) = g_{\sigma_f(i)} k_{\sigma_h \sigma_f(i)} x_{\sigma_h \sigma_f(i)}
\]
and, by \cref{wreathformula},
\begin{align*}
(\bar{k}^{-1}(\sigma_h \bar{g}^{-1}); \sigma_h \sigma_f) (x_i) &= ((\bar{k}^{-1}(\sigma_h \bar{g}^{-1}))_{\sigma_h \sigma_f(i)})^{-1} x_{\sigma_h \sigma_f(i)} \\ &= (k_{\sigma_h \sigma_f(i)}^{-1}(\sigma_h \bar{g}^{-1})_{\sigma_h \sigma_f(i)})^{-1} x_{\sigma_h \sigma_f(i)}\\ &= (\sigma_h \bar{g})_{\sigma_h \sigma_f(i)}k_{\sigma_h \sigma_f(i)}x_{\sigma_h \sigma_f(i)}\\ &= g_{\sigma_{h}^{-1}\sigma_h \sigma_f(i)}k_{\sigma_h \sigma_f(i)}x_{\sigma_h \sigma_f(i)}\\ &= g_{\sigma_f(i)}k_{\sigma_h \sigma_f(i)}x_{\sigma_h \sigma_f(i)}.
\end{align*}

Given a group homomorphism $s \colon G \rightarrow K$ and a $G$-set $X$, we may functorially form the left $K$-set
\[
K \times_G X,
\]
which is the quotient of $K \times X$ by the relation $(ks(g),x) \sim (k,gx)$. This is a $K$-set through the left action of $K$ on itself. The operation $K \times_G -$ is left adjoint to the functor from $K$-sets to $G$-sets given by viewing a $K$-set as a $G$-set through $s$. The $K$-set
\[
K \times_G (\uG \times \um)
\]
is a disjoint union of $m$ $K$-torsors and the unit of the adjunction is a canonical map of $G$-sets
\[
\eta \colon \uG \times \um \rightarrow K \times_G (\uG \times \um).
\]
Given a set of generators $\bar{x}$ of $\uG \times \um$, $\eta(\bar{x})$ is a set of generators of $K \times_G (\uG \times \um)$ and this induces the commutative diagram
\begin{equation} \label{naturality}
\xymatrix{\Aut_G(\uG \times \um) \ar[r]^-{\bar{x}^*} \ar[d] & G \wr \Sigma_m \ar[d] \\ \Aut_K(K \times_G (\uG \times \um)) \ar[r]^-{\eta(\bar{x})^*} & K \wr \Sigma_m.}
\end{equation}

Although the isomorphism of Equation \ref{wreathiso} depends on a choice of generators of $\uG \times \um$, the next lemma shows that this dependence is removed by passing to conjugacy classes.

\begin{lemma} \label{lem:1}
There is a canonical isomorphism, natural in the group $G$,
\[
\hom(\lat, G \wr \Sigma_m)_{/\sim} \cong \hom(\lat, \Aut_G(\uG \times \um))_{/\sim}.
\]
\end{lemma}
\begin{proof}
Let $\bar{x}$ and $\bar{y}$ be sets of generators of $\uG \times \um$ as a $G$-set. Then there exists an automorphism of $f$ of $\uG \times \um$ such that $f(x_i) = y_i$. This automorphism induces an automorphism
\[
\Aut_G(\uG \times \um) \lra{c_f} \Aut_G(\uG \times \um)
\]
given by conjugation with $f$ and the following diagram commutes
\[
\xymatrix{\Aut_G(\uG \times \um) \ar[r]^{\bar{x}} \ar[d]_{c_f} & G \wr \Sigma_m \\ \Aut_G(\uG \times \um) \ar[ur]_{\bar{y}}.}
\]
For $\al$ and $\al '$ conjugate elements in $\hom(\lat, \Aut_G(\uG \times \um))$, it suffices to prove that $\bar{x}^*(\al)$ is conjugate to $\bar{y}^*(\al ')$. Assume that $\al ' = g \al g^{-1}$, then the following diagram commutes
\[
\xymatrix{& \Aut_G(\uG \times \um) \ar[d]_-{c_g} \ar[r]^-{\bar{x}} & G \wr \Sigma_m \ar@{-->}[dd] \\ \lat \ar[ur]^{\al} \ar[r]_-{\al '} & \Aut_G(\uG \times \um) \ar[dr]^-{\bar{y}} \ar[d]_-{c_f} & \\ & \Aut_G(\uG \times \um) \ar[r]_-{\bar{x}} & G \wr \Sigma_m.}
\]
The dashed arrow is built by passing through the isomorphisms and is thus given by conjugation. Thus $\bar{x}^*(\al)$ is conjugate to $\bar{y}^*(\al ')$.

The isomorphism is natural by Diagram \ref{naturality}.
\end{proof}

A map $\lat \rightarrow \Aut_G(\uG \times \um)$ corresponds to an $\lat$-set structure on $\uG \times \um$ in the category of $G$-sets. By adjunction, this corresponds to an $\lat \times G$-set of order $|G| \times m$ with a free $G$-action.

\begin{lemma} \label{lem:2}
By adjunction, there is a canonical isomorphism that is natural in the group $G$ between
\[
\hom(\lat, \Aut_G(\uG \times \um))_{/\sim}
\]
and the set of isomorphism classes of $\lat \times G$-sets of order $|G| \times m$ with a free $G$-action.
\end{lemma}

The projection map $\uG \times \um \rightarrow \um$ induces 
\[
\Aut_G(\uG \times \um) \rightarrow \Aut(\um).
\]
Thus it makes sense to define 
\[
\hom(\lat, \Aut_G(\uG \times \um))^{\trans}_{/\sim}
\]
to consist of the conjugacy classes that induce a transitive $\lat$-action on $\um$ under the projection.  These, in turn, give rise to isomorphism classes of transitive $\lat \times G$-sets with a free $G$-action under the isomorphism of the lemma above. Transitive $\lat$-actions on $\um$ only exist when $\um$ has cardinality $p^k$ for some $k$. In the transitive case we will write $\pk$ instead of $\um$.

As usual, isomorphism classes of transitive $\lat \times G$-sets of order $|G| \times p^k$ correspond (canonically) to conjugacy classes of index $|G| \times p^k$ subgroups of $\lat \times G$. The conjugacy classes of subgroups giving rise to a transitive $\lat \times G$-set with a free $G$-action have a nice description.

\begin{lemma} \label{lem:3}
There is a canonical isomorphism, natural in the group $G$, between conjugacy classes of index $|G| \times p^k$ subgroups $M$ of $\lat \times G$ such that $(\lat \times G)/M$ has a free $G$-action and pairs 
\[
(\lat_H \subseteq \lat, [\al_H \colon \lat_H \rightarrow G])
\]
consisting of an index $p^k$ sublattice $\lat_H \subset \lat$ and a conjugacy class of maps 
\[
[\al_H \colon \lat_H \rightarrow G].
\] 
\end{lemma}
\begin{proof}
Subgroups $M \subset \lat \times G$ with the property that $(\lat \times G)/M$ has a free $G$-action are in correspondence with subgroups of $\lat \times G$ that intersect $0 \times G \subset \lat \times G$ trivially. 

The data of a subgroup with this property has a very simple description. If $M \subset \lat \times G$ intersects $0 \times G$ trivially, then the composite 
\[
M \hookrightarrow \lat \times G \rightarrow \lat
\] 
is injective, where the last map is the projection: If $m \neq m' \in M$ map to $(l,g) \neq (l,g')$ respectively, then $m(m')^{-1}$ maps to $(0,g(g')^{-1})$ which is not the identity. This contradicts the assumption that the intersection with $0 \times G$ is trivial. 

Thus $M$ may be identified with $\lat_H$ for some $\lat_H \subset \lat$ with $H \subset \qz$ of order $m$. Since $\lat$ is abelian, a conjugacy class of subgroups is a subgroup. Thus a conjugacy class of subgroups of $\lat \times G$ of order $|G| \times p^k$ that intersect $0 \times G$ trivially corresponds to a subgroup $\lat_H \subset \lat$ of index $p^k$ and a conjugacy class of maps $[\al_H \colon \lat_H \rightarrow G]$.

Next we prove that the isomorphism is natural. Let $s \colon G \rightarrow K$ be a group homomorphism. The induced map 
\[
\id \times s \colon \lat \times G \rightarrow \lat \times K
\]
sends conjugacy classes of subgroups to conjugacy classes of subgroups. If $M \subset \lat \times G$ intersect $0 \times G$ trivially, then so does $(\id \times s)(M)$.

Also, $s$ induces the map on pairs $(\lat_H, [\al])$ sending 
\[
(\lat_H, [\al]) \mapsto (\lat_H, [s \al]).
\]

Now, by the construction above of $(\lat_H, \al)$ from a subgroup $M$ that intersect $0 \times G$ trivially, we see that $s(M)$ is sent to $(\lat_H,s \al)$. Conjugating the subgroup $M$ just conjugates $s \al$. 
\end{proof}

By construction, the isomorphism of \cref{lem:1} restricts to a canonical isomorphism
\[
\hom(\lat, G \wr \Sigma_{p^k})^{\trans}_{/\sim} \cong \hom(\lat, \Aut_G(\uG \times \pk))^{\trans}_{/\sim}.
\]
For $[\al] \in \hom(\lat, G \wr \Sigma_{p^k})^{\trans}_{/\sim}$, we may use the previous three lemmas to define
\[
F([\al]) = (H, [\al_H]).
\]

\begin{prop} \label{subs}
The map 
\[
\xymatrix{F \colon \hom(\lat, G \wr \Sigma_{p^k})^{\trans}_{/\sim} \ar[r]^-{\cong} & \Sub_{p^k}(\qz,G),}
\]
given by $F([\al]) = (H, [\al_H])$, is a canonical isomorphism that is natural in $G$.
\end{prop}
\begin{proof}
Combine the canonical isomorphisms of the previous three lemmas.
\end{proof}

It is worth explaining an equivalent way of obtaining $(H, [\al_H])$ from $[\al] \in \hom(\lat, G \wr \Sigma_{p^k})^{\trans}_{/\sim}$. We have a commutative diagram
\[
\xymatrix{\lat_H \ar[r]^{(\beta_i)} \ar[d]_-{q_{H}^*} & G^{p^k} \ar[d] \\ \lat \ar[r]^-{\al} \ar[dr]_-{\pi \al} & G \wr \Sigma_{p^k} \ar[d]^-{\pi} \\ & \Sigma_{p^k},}
\]
defining $p^k$ maps $\beta_i \colon \lat_H \rightarrow G$. A set of generators $\bar{x}$ of $\uG \times \pk$ produces an isomorphism
\[
\bar{x}^* \colon \Aut_G(\uG \times \pk) \rightarrow G \wr \Sigma_{p^k}
\]
and the image of $\lat_H \lra{(q_{H}^*,\beta_i)} \lat \times G$ is the stabilizer of $x_i$ in the transitive $\lat \times G$-set determined by the composite $(\bar{x}^*)^{-1}\al$. By \cref{lem:3}, all of the maps $\beta_i$ are conjugate and we may take $\al_H$ to be any of the $\beta_{i}$'s. 

\begin{cor}\label{factorizationlemma}
We may choose $\al \in [\al] \in \hom(\lat, G \wr \Sigma_{p^k})^{\trans}_{/\sim}$ with the property that it factors through the inclusion
\[
\xymatrix{& (\im \al_H) \wr \Sigma_{p^k} \ar@{^{(}->}[d] \\ \lat \ar[r]_-{\al} \ar@{-->}[ur] & G\wr \Sigma_{p^k}.}
\]
\end{cor}
\begin{proof}
For any $\al \in [\al]$, the $p^k$ maps
\[
\lat_H \rightarrow G
\]
are all conjugate. Thus we may choose elements $g_2, \ldots, g_{p^k} \in G$ conjugating them to the first one. The element $(1, g_2, \ldots, g_{p^k};e) \in G\wr \Sigma_{p^k}$ conjugates $\al$ to a map satisfying the required condition.
\end{proof}

Since the isomorphism class of a finite $\lat \times G$-set is determined by the isomorphism classes of the transitive components of a representative, for $[\al] \in \hom(\lat, G \wr \Sigma_{m})_{/\sim}$ we may define
\[
F([\al]) = \Oplus{i}(H_i, [\al_{H_i}]),
\]
where the pairs $(H_i, [\al_{H_i}])$ are determined by the transitive components of the $\lat \times G$-set associated to $[\al]$.

\begin{prop} \label{sums}
There is a canonical isomorphism, which is natural in $G$,
\[
\xymatrix{F \colon \hom(\lat, G \wr \Sigma_{m})_{/\sim} \ar[r]^-{\cong} & \Sum_{m}(\qz,G)}
\]
given by $F([\al]) = \bigoplus_i (H_i, [\al_{H_i}])$.
\end{prop}

\begin{remark}
There is a more geometric interpretation of the above results, see \cite[Section~3.3.4, Corollary~3.7]{globalgreen}. Conjugacy classes of tuples of commuting elements are the homotopy classes
\[
[B\lat, BG\wr \Sigma_m].
\]
The space $B\lat$ is the $p$-complete torus. The homotopy classes $[B\lat, B\Sigma_m]$ are the (isomorphism classes of) $m$-fold covers of the $p$-complete torus and the homotopy classes $[B\lat, BG \wr \Sigma_m]$ are the (isomorphism classes of) principal $G$-bundles on $m$-fold covers of the $p$-complete torus.
\end{remark}

Now that we have proved the main result of this section, we draw several consequences and prove related results that will be used in the remaining part of the paper.

\begin{lemma}\label{padicsum}
Let $\sum_j a_j p^j$ be the $p$-adic expansion of $m$, then the inclusion
$\Prod{j} \Sigma_{p^j}^{a_j} \subseteq \Sigma_m$
induces a surjection
\[
\Prod{j} \Sum_{p^j}(\qz,G)^{\times a_j} \twoheadrightarrow \Sum_m(\qz,G).
\]
\end{lemma}
\begin{proof}
This is immediate from the definitions.
\end{proof}

\begin{prop} \label{diagonal}
The map $\Delta\colon G \times \Sigma_{p^k} \lra{} G\wr \Sigma_{p^k}$ induces the map
\[
\xymatrix{\Sub_{p^k}(\qz) \times \hom(\lat, G)_{/\sim} \ar[r] & \Sub_{p^k}(\qz,G)}
\]
defined by $(H,[\al:\lat \rightarrow G]) \mapsto (H, [\al q_{H}^*])$.
\end{prop} 
\begin{proof}
Let $\beta\colon \lat \rightarrow \Sigma_{p^k}$ be transitive with the property that $F([\beta]) = H$; we are interested in the composite
\[
\lat \lra{\al \times \beta} G \times \Sigma_{p^k} \rightarrow G \wr \Sigma_{p^k}.
\]
This fits in the commutative diagram
\[
\xymatrix{\lat_H \ar[d]_-{q_{H}^*} \ar[r]^-{\al q_{H}^*} & G \ar[d] \ar[r]^{\Delta} & G^{p^k} \ar[d] \\ \lat \ar[r]^-{\al \times \beta} & G \times \Sigma_{p^k} \ar[dr] \ar[r]^{\Delta} & G \wr \Sigma_{p^k} \ar[d] \\ && \Sigma_{p^k},}
\]
where the top row consists of the kernels of the maps to $\Sigma_{p^k}$. It follows that
\[
[(\Delta(\al \times \beta))_H] = [\al (q_{H}^*)].
\]

\end{proof}

Note that we now have a way of describing conjugacy classes in iterated wreath products of symmetric groups. By \cref{sums}, there is a bijection
\[
\hom(\lat, G \wr \Sigma_{s} \wr \Sigma_{t})_{/\sim} \cong \Sum_{t}(\qz,G \wr \Sigma_{s}).
\]
Conjugacy classes from $\lat_{H_i}$ to $G \wr \Sigma_{s}$ satisfy the same kind of formula, thus we may write
\[
\bigoplus_i(H_i, [\al_{H_i}]) = \bigoplus_i(H_i, \bigoplus_j (K_{i,j},[(\al_{H_i})_{K_{i,j}}])),
\]
where $K_{i,j} \subset \qz/H_i$.

\begin{prop} \label{iteratedwreath}
Let $st = m$, then the natural inclusion  
\[
\nabla \colon G \wr \Sigma_{s} \wr \Sigma_{t} \lra{} G \wr \Sigma_{m}\]
induces the map
\[
\bigoplus_i(H_i, \bigoplus_j (K_{i,j},[(\al_{H_i})_{K_{i,j}}])) \mapsto \bigoplus_{i,j}(T_{i,j},[(\al_{H_i})_{K_{i,j}}]) = \bigoplus_{i,j}(T_{i,j},[\al_{T_{i,j}}]), 
\]
where $T_{i,j}$ is defined as the pullback in the following diagram
\[
\xymatrix{T_{i,j} \ar[d] \ar[r] & \qz \ar[d] \\ K_{i,j} \ar[r] & \qz/H_i.}  
\]
\end{prop}
\begin{proof}
It suffices to show this on $\al \colon \lat \lra{} G\wr \Sigma_{p^j} \wr \Sigma_{p^t}$ with $\al$ transitive and $\al_H:\lat_H \lra{} G \wr \Sigma_{p^j}$ transitive. When $\al$ and $\al_H$ are transitive, $\nabla \al$ is transitive. In this case, we have $K \subset \qz/H$ and the pullback $T$ is the kernel of the composite
\[
\qz \lra{q_H} \qz/H \lra{} (\qz/H)/K.
\]
Now we check that $[(\al_H)_K] = [(\nabla \al)_T]$. By \cref{factorizationlemma}, $\al \in [\al]$ can be chosen so that
\[
\xymatrix{(\lat_{H})_K \ar[r] \ar[d] & \lat \ar[d]_-{\al} & \\ G \ar[r]^-{\Delta} \ar[d]_{=} & G \wr \Sigma_{s} \wr \Sigma_{t} \ar[d]^-{\nabla} \ar[r] & \Sigma_{s} \wr \Sigma_{t} \ar[d] \\ G \ar[r]^-{\Delta} & G \wr \Sigma_{m} \ar[r] & \Sigma_{m}} 
\]
commutes, where $(\lat_H)_K$ is the kernel of the composite $\lat \rightarrow G\wr \Sigma_s \wr \Sigma_t \rightarrow \Sigma_{s} \wr \Sigma_{t}$. Since $\Sigma_{s} \wr \Sigma_{t} \rightarrow \Sigma_{m}$ is injective, $(\lat_H)_K = \lat_T$ and $(\al_H)_K =  \al_T$.
\end{proof}

Let $\lat ' \cong \Z_{p}^n$ be another rank $n$ lattice and let $\qz ' = (\lat ')^*$. Let 
\[
\sigma \colon \qz \lra{\cong} \qz '
\]
be an isomorphism. Precomposition with $\sigma^*$ induces a bijection
\[
\hom(\lat, G \wr \Sigma_{m})_{/\sim} \lra{\cong} \hom(\lat ', G \wr \Sigma_{m})_{/\sim}.
\]
We explain the effect of this map on $\Sum_m(\qz,G)$.
\begin{prop} \label{isoaction}
Precomposition with $\sigma^*$ on $\hom(\lat, G \wr \Sigma_{m})_{/\sim}$ induces the map
\[
\Sum_m(\qz,G) \rightarrow \Sum_m(\qz ',G)
\] 
given by
\[
\bigoplus_i(H_i, [\al_{H_i}]) \mapsto \bigoplus_i( \sigma(H_i), [\al_{H_i}\sigma^{*}_{|_{\lat_{\sigma(H_i)} '}}]).
\]
\end{prop}
\begin{proof}
First we will verify the claim on the subgroups $H_i \subset \qz$. If $\lat_{H_i}$ stabilizes $x \in \um_{\pi \al}$, then $H_i$ is the kernel of $\qz = \lat^* \rightarrow \lat_{H_i}^*$. The stabilizer of $x \in \um_{\pi \al \sigma^*}$ is the preimage of $\lat_{H_i}^*$ under $\sigma^*$. The kernel of the map $(\lat ')^* \rightarrow ((\sigma^*)^{-1} \lat_{H_i})^*$ is $\sigma H_i$.

Now we may assume that we have a transitive map $\lat \lra{\al}  G \wr  \Sigma_{p^k}$ such that $[\pi \al]$ corresponds to $H$. By \cref{factorizationlemma}, we may assume that $\al_{|\lat_{H}} = \Delta \al_{H}$. Since $(\sigma^*)^{-1} \lat_{H} = \lat_{\sigma H}$, the diagram
\[
\xymatrix{\lat_{\sigma H} ' \ar[r] \ar[d]_{\sigma^{*}_{|_{\lat_{\sigma(H)} '}}} & \lat ' \ar[d]^{\sigma^*} \\ \lat_{H} \ar[r] \ar[d]_{\al_{H}} & \lat \ar[d]^{\al} \\ G \ar[r]^-{\Delta} & G \wr \Sigma_{p^k},}
\]
shows that $[\al_{H}]$ is sent to $[\al_{H}\sigma^{*}_{|_{\lat_{\sigma(H)} '}}]$.
\end{proof}

There is an obvious left action of $\Aut(\qz)$ on $\hom(\lat, G \wr \Sigma_{m})_{/\sim}$ given by precomposition with the Pontryagin dual. The previous proposition gives a formula for this action on $\Sum_{m}(\qz,G)$.

\begin{cor} \label{sumsaction}
The left action of $\sigma \in \Aut(\qz)$ on $\Sum_{m}(\qz,G)$ induced by the action of $\Aut(\qz)$ on $\hom(\lat, G \wr \Sigma_{m})_{/\sim}$ is given by 
\[
\bigoplus_i(H_i, [\al_{H_i}]) \overset{\sigma}{\mapsto} \bigoplus_i( \sigma(H_i), [\al_{H_i}\sigma^{*}_{|_{\lat_{\sigma(H_i)}}}]).
\]
\end{cor}

\section{The action of isogenies on $C_0$}\label{sec:isogenies}
Let $\Gamma$ be a height $n$ formal group over a perfect field $\kappa$ of characteristic $p$. In this section we introduce three rings depending on this data: the Lubin--Tate ring $\E$, the Drinfeld ring $\D$, and the rationalized Drinfeld ring $C_0$.

The Lubin--Tate ring $\E$, which is non-canonically isomorphic to $W(\kappa)\powser{u_1, \ldots, u_{n-1}}$, first appears in the study of formal groups in \cite{lubintate}. The Lubin--Tate moduli problem associates to a complete local ring $R$ with residue field $\kappa$ the groupoid $\Def_{\Gamma}(R)$ with objects deformations of $\Gamma$ to $R$ and morphisms $\star$-isomorphisms, see \cite[Section 4]{Notes}. The main theorem of Lubin and Tate produces an equivalence of groupoids 
\[
\Spf(\E)(R) \simeq \Def_{\Gamma}(R).
\] 
The ring $\E$ carries the universal deformation $\G$ of $\Gamma$.

Let $D_k$ be the Drinfeld ring of full level-$k$ structures on $\G$, which was introduced in \cite[Section~4B]{drinfeldell1}. Note that Drinfeld omits the word ``full" as all of his level structures are full. The ring $D_k$ is a complete local $\E$-algebra that is finitely generated and free as an $\E$-module, with the property that 
\[
\Spf(D_k) \cong \Level(\qz[p^k], \G),
\]
where $\Level(\qz[p^k], \G)$ is the functor sending a complete local $\E$-algebra $R$ to the set of level structures $\Level(\qz[p^k], R \otimes \G)$. In particular, $D_k$ has been studied by homotopy theorists in \cite[Sections 6.1 and 6.2]{hkr}, \cite[Section 2.4]{isogenies}, \cite[Part 3]{ahs}, and \cite[Section~7]{subgroups}. 

For $H \subset \qz[p^k] \subset \qz$ a finite subgroup, the formal group $\G/H$ may be defined as a deformation over $D_k$ and thus is classified by a map $\E \lra{Q_H} D_k$. It follows that there is a (necessarily unique) $\star$-isomorphism
\[
D_k\tensor[^{Q_H}]{\otimes}{_{\E}} \G \cong (D_k \otimes_{\E} \G)/H.
\]
Let 
\[
\D = \Colim{k} \text{ } D_k = \Colim{k} \text{ } \Gamma \Level(\qz[p^k], \G),
\]
where the maps in the indexing diagram are induced by the precomposition
\[
(\qz[p^k] \lra{} \G) \mapsto (\qz[p^{k-1}] \hookrightarrow \qz[p^k] \lra{} \G).
\]
We define\footnote{We are treating $\D$ as an ind-object since the colimit is not a complete local ring.} $\hom(\D, R) = \lim_k \hom(D_k, R)$, an element in this set being an infinite compatible family of level structures. 

There are actions of several large groups on $\D$, see \cite[Section 1.4]{carayolltt}. There is an action of $\Aut(\qz)$ on $\D$ that plays a prominent role in the theorems in~\cite{hkr}. This action is given by precomposing a level structure $\qz \hookrightarrow \G$ with the given automorphism of $\qz$. Now let $\Isog(\qz)$ be the monoid of endoisogenies (endomorphisms with finite kernel) of $\qz$. Note that $\Isog(\qz)$ is contravariantly isomorphic to $M_{n}^{\det \neq 0}(\Z_p)$, the monoid of $n\times n$ matrices with coefficients in $\Z_p$ that have nonzero determinant. We will extend the opposite of the action of $\Aut(\qz)$ on $\D$ to an action of $\Isog(\qz)$ on $\D$. After proving this result, the authors discovered that it seems to be well-known. Because we were unable to locate a proof in the literature, we include the argument.

\begin{prop} \label{action}
The action of the group $\Aut(\qz)$ by $\E$-algebra maps on $\D$ given by precomposing a level structure with the inverse automorphism extends to an action of $\Isog(\qz)$ by ring maps. 
\end{prop}  

\begin{proof}
Fix an isogeny $\phi_H \colon \qz \lra{} \qz$ with kernel $H$. We want to construct a map 
\[
\D \lra{\phi_H} \D
\] 
by mapping to a cofinal subcategory of the diagram category. For $H \subseteq \qz[p^k] \subset \qz$, fix $k' \geq k$ such that $\qz[p^k] \subseteq \phi_H(\qz[p^{k'}])$. Consider the following diagram
\[
\xymatrix{\E \ar[r]^{Q_H} \ar[d] & D_k \ar[d] & \\ D_k \ar@{-->}[r] & D_{k'} \ar[r] & R,}  
\]
where $Q_H$ is the map classifying the deformation $\G/H$. We would like to construct the dashed map. 

The map $D_{k'} \lra{} R$ is a map of $\E$-algebras for the standard $\E$-algebra structure on $D_{k'}$. Thus we have a level structure $\qz[p^{k'}] \hookrightarrow R \otimes \G$. Taking the quotient by $H$ gives the level structure
\[
\qz[p^{k'}]/H \cong \phi_H(\qz[p^{k'}]) \hookrightarrow (R\otimes \G)/H.
\]
We put a different $\E$-algebra structure on $R$ by using $Q_H$. We abuse notation and denote the composite $\E \lra{Q_H} D_{k'} \lra{} R$ by $Q_H$. Recall that we have the canonical isomorphism $(R \otimes_{\E} \G)/H \cong R \tensor[^{Q_H}]{\otimes}{_{\E}} \G$. Precomposition with the inclusion $\qz[p^{k}] \subseteq \phi_H(\qz[p^{k'}])$ gives
\[
\qz[p^{k}] \hookrightarrow R \tensor[^{Q_H}]{\otimes}{_{\E}} \G,
\]
which is classified by a map $D_k \lra{} R$ making the whole diagram commute. Now take $R = D_{k'}$. 

To check compatibility of the maps, it suffices to consider the following situation. Let $l > k$ and $l' > k'$ so that $\qz[p^l] \subset \phi_H(\qz[p^{l'}])$, then the commutative diagram
\[
\xymatrix{\qz[p^{k'}] \ar@{^{(}->}[r] \ar@{^{(}->}[d] & D_{l'}\otimes_{\E}\G \ar[d]^{=} \\ \qz[p^{l'}] \ar@{^{(}->}[r] & D_{l'} \otimes_{\E} \G }
\]
gives rise to a commutative diagram
\[
\xymatrix{\qz[p^k] \ar@{^{(}->}[r] \ar@{^{(}->}[d] & \phi_H(\qz[p^{k'}]) \ar@{^{(}->}[r] \ar@{^{(}->}[d] & (D_{l'} \otimes_{\E} \G)/H  \ar[r]^{\cong} \ar[d]^{=} & D_{l'} \tensor[^{Q_H}]{\otimes}{_{\E}} \G \ar[d]^{=} \\
\qz[p^l] \ar@{^{(}->}[r] & \phi_H(\qz[p^{l'}]) \ar@{^{(}->}[r] & (D_{l'} \otimes_{\E} \G)/H  \ar[r]^{\cong} & D_{l'} \tensor[^{Q_H}]{\otimes}{_{\E}} \G.}
\]
This implies that the diagram
\[
\xymatrix{\E \ar[r]^{Q_H} \ar[d] & D_k \ar[d] \\ D_k \ar[d] \ar@{-->}[r] & D_{k'} \ar[d] \\ D_l \ar@{-->}[r] & D_{l'}}
\]
commutes. We can now make precise the (more informal sounding) statement that $\phi_H$ sends the universal level structure $\qz \hookrightarrow \D \otimes_{\E} \G$ to the composite
\[
\qz \lra{\psi_{H}^{-1}} \qz/H \lra{} (\D \otimes_{\E} \G)/H \lra{\cong} \D \tensor[^{Q_H}]{\otimes}{_{\E}} \G.
\]

Finally we show that this action extends the action of $\Aut(\qz)$ on $\D$ by precomposition with the inverse automorphism. Indeed, by the above, the automorphism $\phi_e \in \Aut(\qz)$ sends the universal level structure $\qz \hookrightarrow \D \otimes_{\E} \G$ to the composite 
\[
\qz \lra{\phi_{e}^{-1}} \qz \lra{} \D \otimes_{\E} \G,
\]
because in this case $\psi_{e} = \phi_e$, due to the diagram
\[
\xymatrix{\qz \ar[r]^{\phi_e} \ar[dr]_-{q_e=id} & \qz  \\ & \qz/e = \qz \ar[u]_{\psi_e}.}
\]
\end{proof}

We note a corollary of the proof.

\begin{cor}
Given an endoisogeny $\phi_H$ of $\qz$ with kernel $H \subset \qz$, the following diagram commutes:
\[
\xymatrix{\qz \ar[r]^{\phi_H} \ar[dr]_{q_H} & \qz \ar@{^{(}->}[r] & \D \tensor[^{Q_H}]{\otimes}{_{\E}} \G  \\
& \qz/H \ar@{^{(}->}[r] \ar[u]_-{\psi_H} & (\D \otimes_{\E} \G)/H. \ar[u]^-{\cong}}
\]
\end{cor}

The ring 
\[
C_0 = p^{-1}\D = \Q \otimes \D
\] 
was introduced in \cite[Section 6.2]{hkr} (where it is called $L(E^0)$) to serve as suitable coefficients for the codomain of the generalized character map. By the previous proposition, $C_0$ is acted on by $\Isog(\qz)$. Hopkins, Kuhn, and Ravenel \cite[Proposition 6.6]{hkr} showed that $C_{0}^{\Aut(\qz)} \cong p^{-1} \E$. It seems reasonable to conjecture that $C_{0}^{\Isog(\qz)} \cong \Q_p$. The ring $C_0$ also has an algebro-geometric description as it carries the universal isomorphism (of \pdiv groups) $\qz \lra{\cong} \G$.

\begin{remark}\label{rmk:hkraction}
The action defined above in \Cref{action} is a right action of $\Isog(\qz)$ on $\D$. The monoid $\Isog(\qz)$ is dual to $M_{n}^{\det \neq 0}(\Z_p)$, which acts on $\D$ on the left. Given $x \in C_0$ and an isogeny $\phi_H$, we will write $\phi_{H}^*(x)$ for the action by the Pontryagin dual. 
One further remark seems necessary. In \cite{hkr} the inverse action is defined and they write it as a left action even though it is a right action. This could cause confusion.
\end{remark}

\begin{example}
Take $n=1$ and the height $1$ formal group $\hat{\G}_m$. In this case the Drinfeld ring $D_k$ is just $\Z_p$ adjoin all primitive $p^k$th roots of unity. The isogeny $[p^k]$ acts trivially on $\D$. This can be seen in many ways. To see it using the proof above note that $Q_H$ must be the standard algebra structure since $\E \cong \Z_p$. Thus $(R \otimes \hat{\G}_m)/H \cong R \otimes \hat{\G}_m$ canonically. Now the level structure $\Lk \lra{} R \otimes \hat{\G}_m$ is precisely the restriction to $\Lk$ of the one we began with. Since every isogeny has the form $up^k$ for some unit (automorphism) $u$, we see that
\[
\D^{\Isog(\Q_p/\Z_p)}  = \D^{\Aut(\Q_p/\Z_p)} = \Z_p.
\]
\end{example} 

\begin{example} \label{ex:adams1}
In fact, the isogeny $[p^k]\colon \qz \lra{} \qz$ always acts by isomorphisms on $\D$. When $\kappa \subset \F_{p^n}$ and $\G$ is the universal deformation of the Honda formal group, $[p^k]$ acts by the identity. This is due to the fact that, in this case, there is a $\star$-isomorphism
\[
(D_k \otimes_{\E} \G)/\qz[p^k] \cong D_k \otimes_{\E} \G
\]
covering the $nk$-fold Frobenius (which is the identity) on $\kappa$. 
\end{example}

\section{Multiplicative natural transformations on class functions}\label{sec:rationaltrafo}
The goal of this section is to construct a collection of operations on class functions, 
\[
\Po_{m}^{\phi}\colon Cl_n(G,C_0) \lra{} Cl_n(G\wr \Sigma_{m},C_0),
\]
that are multiplicative, non-additive, natural in the finite group $G$. These operations combine the action of isogenies on $C_0$ with the description of conjugacy classes in wreath products given in \Cref{sec:grouptheory}. In \Cref{sec:proofgeneral}, we will show that the operations constructed in this section are compatible with the total power operation in $E$-theory.

Recall that $\lat = \Z_{p}^n$ and that $\hom(\lat,G)_{/\sim}$ is the set of conjugacy classes of maps from $\lat$ to $G$ (see \Cref{sec:notation}). Given an isogeny $\phi_H \colon \qz \rightarrow \qz$ with kernel $H \subset \qz$, there is an action map $\phi_{H}^*\colon C_0 \lra{} C_0$. Also recall that $\psi_H\colon\qz/H \lra{\cong} \qz$ is the isomorphism induced by $\phi_H$ and that $\psi_{H}^*\colon \lat \lra{} \lat_H$ is the Pontryagin dual.

\begin{definition}
For a finite group $G$ and a subgroup $H \subset \qz$, define $Cl_{n}^{H}(G, C_0)$ to be the set of $C_0$-valued functions on $\hom(\lat_H,G)_{/\sim}$. When $H$ is the trivial group we will abbreviate this to $Cl_n(G,C_0)$. Recall that $C_0$ depends on $n$ and that $Cl_{n}^{H}(G, C_0)$ depends on the prime $p$ even though it is not part of the notation.
\end{definition}

\begin{prop}\label{isogtrafo}
Given an isogeny $\phi_H\colon \qz \lra{} \qz$ with kernel $H$ we may produce a natural map of rings
\[
(-)^{\phi_H} \colon Cl_n(G,C_0) \lra{} Cl_{n}^H(G,C_0)
\]
defined by $f^{\phi_H}([\al]) = \phi_{H}^* f ([\al \psi_{H}^*])$.
\end{prop}
\begin{proof}
Note that the map is well defined.

To prove naturality, let $K \lra{i} G$ be a map of finite groups. This induces restriction maps $\Res\colon Cl_{n}(G,C_0) \lra{} Cl_{n}(K,C_0)$ and $\Res^H\colon Cl_{n}^H(G,C_0) \lra{} Cl_{n}^H(K,C_0)$. Now we show that $\Res^H \phi_H = \phi_H \Res$. Let $\al\colon \lat_H \lra{} K$, then we have 
\begin{align*}
\Res^H (f^{\phi_H}) ([\al]) &= f^{\phi_H}([i \al]) \\
&= \phi_{H}^* f([i \al \psi_{H}^*])\\
&= \phi_{H}^* \Res(f)([\al \psi_{H}^*])\\
&= \Res(f)^{\phi_H}([\al]).
\end{align*}

For completeness, we show the map is a ring map. Let $f,g \in Cl_n(G,C_0)$, then
\begin{align*}
(f+g)^{\phi_H}([\al]) &= \phi_{H}^*(f+g)([\al \psi_{H}^*]) \\
&= \phi_{H}^*(f([\al \psi_{H}^*]) + g([\al \psi_{H}^*])) \\
&= \phi_{H}^*f([\al \psi_{H}^*]) + \phi_{H}^*g([\al \psi_{H}^*]) \\
&= f^{\phi_H}([\al]) + g^{\phi_H}([\al])
\end{align*}
and
\begin{align*}
(fg)^{\phi_H}([\al]) &= \phi_{H}^*(fg)([\al \psi_{H}^*])\\
&= \phi_{H}^*(f([\al \psi_{H}^*])g([\al \psi_{H}^*]))\\
&= \phi_{H}^*f([\al \psi_{H}^*])\phi_{H}^*g([\al \psi_{H}^*])\\
&= f^{\phi_H}([\al])g^{\phi_H}([\al]).
\end{align*}
\end{proof}

\begin{cor}
When restricted to $\Aut(\qz) \subset \Isog(\qz)$ the maps defined above are endomorphisms of $Cl_n(G,C_0)$. This is the inverse of the action of $GL_n(\Z_p)$ on class function described by \cite{hkr}.
\end{cor}
\begin{proof}
The important thing to note is that $q_e = id\colon \qz \rightarrow \qz$, so we have a triangle
\[
\xymatrix{\qz \ar[r]^{\phi_e} \ar[dr]_-{id} & \qz  \\ & \qz/e = \qz \ar[u]_{\psi_e}}
\]
factoring the identity map. Thus $\psi_e = \phi_{e}$ and \Cref{isogtrafo} produces a map
\[
Cl_n(G,C_0) \lra{} Cl_{n}(G,C_0)
\]
and
\[
f^{\phi_e}([\al]) = \phi_{e}^* f ([\al \psi_{e}^*]) = \phi_{e}^* f ([\al \phi_{e}^*]),
\]
which is the action described in \cite{hkr} after taking into account \Cref{rmk:hkraction}.
\end{proof}

\begin{remark}
We will often use the canonical isomorphism
\[
Cl_n(G,C_0) \otimes_{C_0} Cl_n(K,C_0) \lra{\cong} Cl_n(G \times K,C_0),
\]
which follows immediately from the definition of $Cl_n(-,C_0)$.
\end{remark}

Our goal is to construct a family of multiplicative natural transformations 
\[
Cl_n(-,C_0) \implies Cl_n(-\wr \Sigma_{m},C_0)
\]
for all $m\geq 0$.

Let $\Sub(\qz)$ be the set of finite subgroups of $\qz$. We may produce a finite subgroup of $\qz$ from an isogeny $\al\colon\qz \lra{} \qz$ by taking the kernel of $\al$. The induced surjection 
\[
\Isog(\qz) \lra{} \Sub(\qz)
\]
makes $\Isog(\qz)$ into a principal $\Aut(\qz)$-bundle over $\Sub(\qz)$, and we denote its set of sections by $\Gamma(\Sub(\qz), \Isog(\qz))$.

Fix a section $\phi \in \Gamma(\Sub(\qz), \Isog(\qz))$ and let $\phi_H = \phi(H)$. \Cref{sums} and \Cref{action} may be combined to construct a map 
\[
\Po_{m}^{\phi}\colon Cl_n(G,C_0) \lra{} Cl_n(G\wr \Sigma_{m},C_0).
\]
For $f \in Cl_n(G,C_0)$, we define
\begin{align*}
\Po_{m}^{\phi}(f)(\bigoplus_i(H_i,[\al_i])) &= \Prod{i}f^{\phi_{H_i}}([\al_i]) \\
&= \Prod{i} \phi_{H_i}^* f([\al_i\psi_{H_i}^*]).
\end{align*}
Let $F \colon \hom(\lat,G\wr \Sigma_m)_{/\sim} \lra{\cong} \Sum_m(\qz,G)$ be the isomorphism of \Cref{sums}. For a conjugacy class $[\al] \in \hom(\lat,G\wr \Sigma_m)_{/\sim}$, we set
\begin{align*}
\Po_{m}^{\phi}(f)([\al]) &= \Po_{m}^{\phi}(f)(F([\al])) \\
&= \Po_{m}^{\phi}(f)(\bigoplus_i(H_i, [\bar{\al}_i])).
\end{align*}

\begin{prop}
Each section $\phi \in \Gamma(\Sub(\qz),\Isog(\qz))$ gives a natural multiplicative transformation
\[
\Po^{\phi}_m\colon Cl_n(-,C_0) \implies Cl_n(-\wr \Sigma_{m},C_0).
\]
\end{prop}
\begin{proof}
Multiplicativity follows from \Cref{isogtrafo}. Indeed, for $f,g \in Cl_n(G,C_0)$ we have
\begin{align*}
\Prod{i}f^{\phi_H}([{\al}_i])\Prod{i}g^{\phi_H}([{\al}_i]) &= \Prod{i}f^{\phi_H}([{\al}_i])g^{\phi_H}([{\al}_i]) \\
&= \Prod{i}(fg)^{\phi_H}([{\al}_i]).
\end{align*}

A map of groups $K \lra{j} G$ induces a map on conjugacy classes $j_*\colon\hom(\lat, K)_{/\sim} \lra{} \hom(\lat,G)_{/\sim}$ and a restriction map $Cl_n(G,C_0) \lra{\Res} Cl_n(K,C_0)$. We show that the diagram
\[
\xymatrix{Cl_n(G,C_0) \ar[r]^-{\Po^{\phi}_m(G)} \ar[d]^{\Res} & Cl_n(G\wr \Sigma_{p^k},C_0) \ar[d]^{\Res} \\ Cl_n(K,C_0) \ar[r]^-{\Po^{\phi}_m(K)} & Cl_n(K \wr \Sigma_{p^k},C_0)}
\]
commutes. We have a commutative diagram
\[
\xymatrix{K\wr \Sigma_{p^k} \ar[r] \ar[d] & G \wr \Sigma_{p^k} \ar[d] \\ \Sigma_{p^k} \ar[r]^{=} & \Sigma_{p^k}.}
\]
Thus 
\[
j_*(\bigoplus_i (H_i, [\al_i])) = \bigoplus_i (H_i, j_*[\alpha_i]),
\]
and this implies the result.

\end{proof}

\begin{definition}\label{def:rationalpo}
We call the map 
\[
\Po^{\phi}_{m}\colon Cl_n(-,C_0) \implies Cl_n(-\wr \Sigma_{m},C_0)
\] 
the pseudo-power operation associated to $\phi \in \Gamma(\Sub(\qz), \Isog(\qz))$.
\end{definition}

\begin{remark}
It is important to note that $\Po_{1}^{\phi}$ is not necessarily the identity. It depends on the choice of automorphism $\phi_e \in \Isog(\qz)$.
\end{remark}

Note that there is an action of $\Aut(\qz)$ on $\Gamma(\Sub(\qz), \Isog(\qz))$ given by multiplication on the left. For $\gamma \in \Aut(\qz)$, $\phi \in \Gamma(\Sub(\qz), \Isog(\qz))$, and $H \in \Sub(\qz)$, we have 
\[
(\gamma \phi)_H = \gamma (\phi_H)\colon \qz \rightarrow \qz.
\]
This action is compatible with the action of $\Aut(\qz)$ on class functions.
\begin{prop}
For $\gamma \in \Aut(\qz)$ and $\phi \in \Gamma(\Sub(\qz), \Isog(\qz))$, we have
\[
\Po_{m}^{\phi}(f^{\gamma}) = \Po_{m}^{\gamma \phi}(f).
\]
\end{prop}
\begin{proof}
It suffices to evaluate on conjugacy classes. By multiplicativity it suffices to check the claim on elements of $\Sub_m(\qz,G)$. Let $(H,[\al]) \in \Sub_m(\qz,G)$, then
\begin{align*}
\Po_{m}^{\phi}(f^{\gamma})(H, [\al]) &= \phi_{H}^* f^{\gamma}([\al \psi_{H}^*]) \\
&= \phi_{H}^* \gamma^* f ([\al \psi_{H}^* \gamma^*]) \\
&= \Po_{m}^{\gamma \phi}(f)(H,[\al]).
\end{align*}
\end{proof}

\section{A total power operation on class functions}\label{sec:rationalpo}
In this section we study the basic properties of the pseudo-power operations from the previous section. Most importantly, we find that if $\phi$ satisfies a certain combinatorial identity, then $\Po_{m}^{\phi}$ equips $Cl_n(-,C_0)$ with the structure of a global power functor in the sense of \cite[Section 4]{globalgreen}. We prove that there are sections $\phi$ satisfying this identity when $n=1,2$.

The functor $Cl_n(-,C_0)$ is a global Green functor in the sense of \cite[Section 3]{globalgreen}. Transfers, restrictions, and their properties are discussed in Section 3 of \cite{globalgreen} and in Theorem D of \cite{hkr}. If $K \subset G$, $f \in Cl_n(K,C_0)$, and $\al \colon \lat \lra{} G$, then the formula for the transfer is 
\[
\Tr(f)([\al]) = \sum_{gK \in (G/K)^{\im \al}}f([g \al g^{-1}]).
\]

We begin by proving an analogue of Proposition~\upperRomannumeral{8}.1.1 from \cite{bmms}.

Let $i,j \geq 0$ and consider the following maps: 
\begin{align} \label{maps}
\nabla &\colon G \wr \Sigma_i \wr \Sigma_j \lra{} G \wr \Sigma_{ij},  \\
\Delta &\colon G \lra{} G \wr \Sigma_m,\nonumber \\
\Delta_{i,j} &\colon G\wr (\Sigma_i \times\Sigma_j) \lra{} G \wr \Sigma_{i+j}, \nonumber \\
\delta &\colon (G \times K) \wr \Sigma_{m} \lra{} (G \wr \Sigma_m) \times (K \wr \Sigma_m). \nonumber
\end{align}
Let 
\[
\Tr_{i,j} \colon Cl_n(\Sigma_i \times \Sigma_j,C_0) \lra{} Cl_n(\Sigma_m,C_0)
\]
be the transfer along $\Sigma_i \times \Sigma_j \subset \Sigma_m$, where $i+j = m$.

\begin{prop}
Let $\phi \in \Gamma(\Sub(\qz),\Isog(\qz))$ and $m,l \ge 0$, then the pseudo-power operations associated to $\phi$ makes the following diagrams commute:
\begin{enumerate}
\item
\[
\xymatrix{Cl_n(G, C_0) \ar[r]^-{\Po^{\phi}_m \times \Po^{\phi}_l} \ar[d]_-{\Po^{\phi}_{m+l}} & Cl_n(G\wr \Sigma_m,C_0) \times Cl_n(G \wr \Sigma_l,C_0) \ar[d]^-{\times} \\ Cl_n(G \wr \Sigma_{m+l}, C_0) \ar[r]^-{\Delta_{m,l}^*} & Cl_n(G \wr (\Sigma_m \times \Sigma_l),C_0),}
\]
\item
\[
\xymatrix{Cl_n(G, C_0) \times Cl_n(K, C_0) \ar[r]^-{\Po^{\phi}_m \times \Po^{\phi}_m} \ar[dr]_-{\Po^{\phi}_m} & Cl_n(G \wr \Sigma_m,C_0) \times Cl_n(K \wr \Sigma_m,C_0) \ar[d]^-{\delta^*} \\ & Cl_n((G \times K) \wr \Sigma_{m}),}
\]
\item
\[
\xymatrix{Cl_n(G, C_0) \ar[r]^-{\Po_{m}^{\phi}} \ar[dr]_-{(\Po_{1}^{\phi})^{\times m}} & Cl_n(G \wr \Sigma_m,C_0) \ar[d]^-{\Delta^*} \\ & Cl_n(G^m,C_0),}
\]
\item
\[
\Po^{\phi}_m(f_1+f_2) = \Po^{\phi}_m(f_1) + \Po^{\phi}_m(f_2) + \sum_{j=1}^{m-1} \Tr_{j, m-j}(\Po^{\phi}_j(f_1) \times \Po^{\phi}_{m-j}(f_2)).
\]
\end{enumerate}
\end{prop}
\begin{proof}
These are all standard. We elaborate on the last formula. 

Let $\bigoplus_i(H_i, [\al_i]) \in \Sum_{m}(\qz,G)$. Note that
\[
\Po^{\phi}_m(f_1+f_2)(\bigoplus_i(H_i, [\al_i])) = \Prod{i}(\phi_{H_i}^*f_1([\al \psi_{H_i}^*]) + \phi_{H_i}^*f_2([\al \psi_{H_i}^*]))
\]
and expand this into a sum (without simplifying). Assume that the sum of subgroups is indexed by $S$ so that
\[
\bigoplus_iH_i = \Oplus{i \in S} H_i.
\]
The formula is completely combinatorial. Let 
\[
Z_j = \{T \subset S|\sum_{i \in T} |H_i| = j\}. 
\]
By a standard argument the transfer equals
\[
\Tr_{j, m-j}(\Po^{\phi}_j(f_1) \times \Po^{\phi}_{m-j}(f_2))(\Oplus{i \in S} (H_i,[\al_i])) = \sum_{T \in Z_j} \Po^{\phi}_j(f_1)(\Oplus{i \in T}(H_i, [\al_i])) \Po^{\phi}_{m-j}(f_2)(\Oplus{i \in S \setminus T}(H_i, [\al_i])). 
\]
As $j$ varies this hits exactly the summands from the expanded sum.
\end{proof}

Let $J \subset Cl_n(G \wr \Sigma_{p^k},C_0)$ be the ideal generated by the image of the transfer along the maps
\[
G \wr (\Sigma_i \times \Sigma_j) \lra{} G \wr \Sigma_{p^k}
\]
for $i,j > 0$ and $i+j = p^k$. In the special case that $G=e$, we will let $I \subset Cl_n(\Sigma_{p^k}, C_0)$ be the ideal generated by the transfer along the maps $\Sigma_i \times \Sigma_j \subset \Sigma_{p^k}$.

We now give a description of $J$ in the spirit of \Cref{sec:grouptheory}.
\begin{prop} \label{transferideal}
The ideal $J$ consists of the factors of 
\[
Cl_n(G\wr \Sigma_{p^k},C_0) = \Prod{\Sum_{p^k}(\qz,G)}C_0
\] 
corresponding to the sums with more than one summand.
\end{prop}
\begin{proof}
This follows immediately from the formula for the transfer. If there is a lift
\[
\xymatrix{& G \wr (\Sigma_{i} \times \Sigma_j) \ar[d] \\ \lat \ar[r]^-{\al} \ar@{-->}[ur] & G\wr \Sigma_{p^k}}
\]
up to conjugacy, then $F([\pi \al])$ must have more than one subgroup. Now the transfer of the characteristic class function concentrated on the lift is a generator of the factor corresponding to $[\al]$. 
\end{proof}

\begin{remark}
We could define $J$ for $m \neq p^k$, but the previous proposition shows that, in this case, $J = Cl_n(G\wr \Sigma_m,C_0)$.
\end{remark}

We define some simplifications of $\Po_{m}^{\phi}$. Consider
\[
\Po_{p^k}^{\phi}/J\colon Cl_n(G, C_0) \lra{} Cl_n(G\wr \Sigma_{p^k},C_0)/J,
\]
\[
P^{\phi}_{m} = \Delta^* \Po_{m}^{\phi} \colon Cl_n(G, C_0) \lra{} Cl_n(G,C_0) \otimes_{C_0} Cl_n(\Sigma_{m},C_0),
\]
where $\Delta \colon G \times \Sigma_{p^k} \rightarrow G \wr \Sigma_{p^k}$, and
\[
P^{\phi}_{p^k}/I\colon Cl_n(G, C_0) \lra{} Cl_n(G,C_0) \otimes_{C_0} Cl_n(\Sigma_{p^k},C_0)/I.
\]

\begin{cor}
The maps $\Po_{p^k}^{\phi}/J$ and $P_{p^k}^{\phi}/I$ are both ring maps.
\end{cor}
\begin{proof}
This follows immediately from the definition of $\Po^{\phi}_m$ and \cref{transferideal}.
\end{proof}

\begin{remark}\label{drinfeldringproducts}
By construction, the pseudo-power operation $\Po_{m}^{\phi}$ restricts to a map
\[
Cl_n(G,\D) \rightarrow Cl_n(G \wr \Sigma_m, \D).
\]
Instead of defining $\Po_{p^k}^{\phi}/J$ and $P_{p^k}^{\phi}/I$ as above, we could have defined them as the rationalization of a map with domain $Cl_n(G, \D)$. This is not true of $\Po_{m}^{\phi}$ as it is not additive.
\end{remark}

\begin{example}
Let $(H,[\al]) \in \Sub_{p^k}(\qz,G)$. We compute the effect of $P^{\phi}_{p^k}/I$ on $Cl_n(G,C_0)$. Using \Cref{diagonal} we see that
\begin{align*}
(P_{p^k}^{\phi}/I)(f)(H,[\al]) &= \Po_{p^k}^{\phi}(f)(H,[\al q_{H}^*]) \\
&= f^{\phi_H}([\al q_{H}^*]) \\
&= \phi_{H}^* f([\al q_{H}^* \psi_{H}^*]) \\
&= \phi_{H}^* f([\al \phi_{H}^*]). 
\end{align*}
\end{example}

\begin{example} \label{adamsops}
When $m = p^{kn}$, there is a further natural simplification of $\Po^{\phi}_m$. If we choose $\phi$ to take $\qz[p^k]$ to $[p^k]\colon \qz \lra{} \qz$, then we may compose $P^{\phi}_{p^{kn}}/I$ further with the map
\[
Cl_n(\Sigma_{p^{kn}},C_0)/I \lra{} C_0
\]
that projects onto the factor corresponding to $\qz[p^k]$. This gives a map
\[
\psi^{p^k}\colon Cl_n(G,C_0) \lra{} Cl_n(G,C_0)
\]
known as the Adams operation corresponding to $\qz[p^k]$. The formula for it can be derived from the previous example:
\[
(P^{\phi}_{p^{kn}}/I)(f)(\qz[p^k],[\al]) = \phi_{\qz[p^k]}^*f([\al \phi_{\qz[p^k]}^*]).
\]
When $\kappa \subset \F_{p^n}$ and $\G$ is the universal deformation of the Honda formal group (following \cref{ex:adams1}), this simplifies to give
\[
f([\al \phi_{\qz[p^k]}^*]).
\]
This map is induced by the map on conjugacy classes $\hom(\lat, G)_{/\sim} \lra{} \hom(\lat, G)_{/\sim}$ sending
\[
[\al\colon\lat \rightarrow G] \mapsto [\lat \lra{p^k} \lat \lra{\al} G],
\]
recovering the usual formula for the Adams operations on class functions when $n=1$. 
\end{example}

We see that we have analogues of all of the parts of Proposition~\upperRomannumeral{8}.1.1 from \cite{bmms} except Part (ii), which is the fundamental relation that a global power functor satisfies (see Section 4 in \cite{globalgreen} as well). Now we classify the sections $\phi$ that give rise to a pseudo-power operation that does satisfy this relation.

\begin{definition}
A section $\phi \in \Gamma(\Sub(\qz), \Isog(\qz))$ is a \emph{power section} if for all $H \subset T \subset \qz$
\[
\phi_T = \phi_{T/H}\phi_H, 
\]
where $T/H = \phi_H(T)$.
\end{definition}

\begin{prop} \label{powersecthm}
The pseudo-power operation $\Po_{m}^{\phi}$ is a global power functor if and only if $\phi$ is a power section.
\end{prop}
\begin{proof}
Assume $\phi$ is a power section. Let $st = m$, we must show that
\[
\xymatrix{Cl_n(G,C_0) \ar[r]^-{\Po^{\phi}_{m}} \ar[d]^-{\Po^{\phi}_{s}} & Cl_n(G \wr \Sigma_{m},C_0) \ar[d]^-{\nabla^*} \\ Cl_n(G \wr \Sigma_{s},C_0) \ar[r]^-{\Po^{\phi}_{t}} & Cl_n(G \wr \Sigma_{s}\wr\Sigma_{t},C_0)}
\]
commutes. 

Recall that $T_{i,j}$ fits into the following commutative diagram
\[
\xymatrix{T_{i,j} \ar[d] \ar[r] & \qz \ar[d] \ar[r]^{\phi_{H_i}} & \qz \\ K_{i,j} \ar[r] & \qz/H_i. \ar[ur]_{\psi_{H_i}}}  
\]
Thus we have an induced map
\[
\psi_{H_i}/K_{i,j} \colon \qz/T_{i,j} \rightarrow \qz/\psi_{H_i}(K_{i,j}).
\]
This fits into the commutative diagram
\begin{equation} \label{bigdiagram}
\xymatrix{\qz \ar[d]^-{q_{H_i}} \ar[dr]^-{\phi_{H_i}} & & \\ \qz/H_i \ar[r]^-{\psi_{H_i}} \ar[d] & \qz \ar[d]_-{q_{\psi_{H_i}(K_{i,j})}} \ar[dr]^-{\phi_{T_{i,j}/H_i} = \phi_{\psi_{H_i}(K_{i,j})}} & \\ \qz/T_{i,j} \ar[r]_-{\psi_{H_i}/K_{i,j}} & \qz/\psi_{H_i}(K_{i,j}) \ar[r]_-{\psi_{T_{i,j}/H_i}} & \qz.}
\end{equation}
Since $\phi$ is a power section, we have the relation
\[
\phi_{T_{i,j}/H_i} \phi_{H_i} = \phi_{T_{i,j}}.
\]
The composite of the left-hand vertical maps is $q_{T_{i,j}}$. Thus we have an equality
\[
\psi_{T_{i,j}/H_i} \psi_{H_i}/K_{i,j} = \psi_{T_{i,j}}.
\]
Applying Pontryagin duality to this equality gives
\[
\psi_{T_{i,j}}^{*} = \psi_{H_i}^*|_{\lat_{\psi_{H_i}(K_{i,j})}} \psi_{T_{i,j}/H_i}^*.
\]
This relation is used in going from line 6 to line 7 of the equalities below.

\Cref{iteratedwreath} shows that a generalized conjugacy class in $G \wr \Sigma_s \wr \Sigma_t$ has the form
\[
\bigoplus_i(H_i, \bigoplus_j (K_{i,j},[(\al_{H_i})_{K_{i,j}}])).
\]

Now we may use the definition of the pseudo-power operation to compute the effect of $\Po^{\phi}_{t}\Po^{\phi}_{s}$ and $\Po^{\phi}_{m}$ on a class function $f \in Cl_n(G,C_0)$. Let $f \in Cl_n(G,C_0)$, then
\begin{align*}
\Po^{\phi}_{t}(\Po^{\phi}_{s}(f))(\bigoplus_i(H_i, \bigoplus_j (K_{i,j},[(\al_{H_i})_{K_{i,j}}]))) &= \Prod{i}(\Po^{\phi}_{s}(f))^{\phi_{H_i}}(\bigoplus_j (K_{i,j},[(\al_{H_i})_{K_{i,j}}])) \\
&=\Prod{i}\phi_{H_i}^*\Po^{\phi}_{s}(f)((\bigoplus_j (K_{i,j},[(\al_{H_i})_{K_{i,j}}]))\psi_{H_i}^*) \\
&=\Prod{i}\phi_{H_i}^*\Po^{\phi}_{s}(f)(\bigoplus_j (\psi_{H_i}(K_{i,j}),[(\al_{H_i})_{K_{i,j}} \psi_{H_i}^*|_{\lat_{\psi_{H_i}(K_{i,j})}}])) \\
&=\Prod{i}\phi_{H_i}^*\Po^{\phi}_{s}(f)(\bigoplus_j (\phi_{H_i}(T_{i,j}),[(\al_{H_i})_{K_{i,j}} \psi_{H_i}^*|_{\lat_{\psi_{H_i}(K_{i,j})}}])) \\
&=\Prod{i}\Prod{j}\phi_{H_i}^*(f)^{\phi_{T_{i,j}/H_i}}([(\al_{H_i})_{K_{i,j}} \psi_{H_i}^*|_{\lat_{\psi_{H_i}(K_{i,j})}}]) \\
&=\Prod{i}\Prod{j}\phi_{H_i}^*\phi_{T_{i,j}/H_i}^*f([(\al_{H_i})_{K_{i,j}} \psi_{H_i}^*|_{\lat_{\psi_{H_i}(K_{i,j})}}\psi_{T_{i,j}/H_i}^*]) \\
&=\Prod{i}\Prod{j}\phi_{T_{i,j}}^*f([(\al_{H_i})_{K_{i,j}} \psi_{T_{i,j}}^*]) \\
&=\Po^{\phi}_{m}(f)(\bigoplus_{i,j}(T_{i,j},[(\al_{H_i})_{K_{i,j}}])). \\
\end{align*}
Going from the second to third line of the equalities we have used \cref{isoaction} applied to $\psi_{H_i}$. 

The reverse implication follows immediately. 
\end{proof}

\begin{example}
When $n=1$, there is an obvious power section: for $\Z/p^k \subset \Q_p/\Z_p$ we let $\phi_{\Z/p^k}$ be the multiplication by $p^k$ map on $\Q_p/\Z_p$.
\end{example}

The authors have been unable to find a general method for producing power sections for all $n$. We now turn to the case of $n=2$, which is important for applications to elliptic cohomology. We write elements of $\Isog(\QZ{2})$ as matrices with entries in $\Z_p$, but view these as acting on $\QZ{2}$.

\begin{lemma} \label{orderp} For each $H \in \Sub_p(\QZ{2})$, there exists a matrix $\phi_{H}$ such that 
\[
\phi_{H}^2 = [p] = \begin{bmatrix} p & 0 \\ 0 & p \end{bmatrix} \text{ and } \phi_{H}(\QZ{2}[p]) = H.
\]
\end{lemma}
\begin{proof}
Associate to the subgroup generated by $(\frac{1}{p},\frac{i}{p})$ the matrix
\[
\begin{bmatrix} -i & 1 \\ p-i^2 & i \end{bmatrix}
\]
and to the subgroup generated by $(0,\frac{1}{p})$ the matrix
\[
\begin{bmatrix} 0 & p \\ 1 & 0 \end{bmatrix}.
\]
\end{proof}

Let us fix endomorphisms for each subgroup of order $p$ as in the above lemma and let $N$ be the submonoid of $\End(\Z_{p}^2)$ generated by the associated matrices $\phi_H$. Note that, if $H$ denotes a finite cyclic subgoup of $\QZ{2}$, then there exists a unique ordered tuple of matrices $(\phi_{K_1},\ldots,\phi_{K_i})$ associated to subgroups $K_j$ of order $p$ such that $H = \ker(\phi_{K_i} \cdots \phi_{K_1})$. The next lemma provides a generalization of this observation to arbitrary finite subgroups $H \subset \QZ{2}$, thereby producing a normal form for $\phi_H$. 

\begin{lemma}\label{lem:powersectionh2}
Suppose $H \subset \QZ{2}$ and let $k$ be the unique natural number with $\QZ{2}[p^k] \subseteq H$ and $\QZ{2}[p^{k+1}] \nsubseteq H$. Choose subgroups $K_{\alpha}$ of order $p$ such that 
\[
H = \ker(\phi_{K_i} \cdots \phi_{K_1}), \qquad \text{(i.e., $|H| = p^i$)}
\]
then the matrix $[p^k]$ divides $\phi_{K_i} \cdots \phi_{K_1}$ in $N$. Moreover, writing $\phi_{K_i} \cdots \phi_{K_1} = [p^k]\phi_{L_s} \cdots \phi_{L_1}$, the $\phi_{L_{\beta}}$ are uniquely determined by $H$. 
\end{lemma}
\begin{proof}
We prove this by induction on the order of $H$. We may assume that 
\[
\QZ{2}[p^k] \nsubseteq H_1 :=  \ker(\phi_{K_{i-1}} \cdots \phi_{K_1}). 
\]
If not, then we are done by induction on the order. 

Thus we are in the situation where $H \cong \Z/p^k \times \Z/p^j$ with $k \leq j$ and $H_1 \cong \Z/p^{k-1} \times \Z/p^j$. By induction, there exist subgroups $L_1,\ldots,L_s$ of order $p$ such that
\[
\phi_{K_{i-1}} \cdots \phi_{K_1} = [p^{k-1}]\phi_{L_s} \cdots \phi_{L_1}.
\]
By definition
\begin{align*}
\phi_{L_s}[p^{k-1}]\phi_{L_{s-1}}\cdots \phi_{L_1}(H) & = [p^{k-1}]\phi_{L_s} \cdots \phi_{L_{1}}(H) \\
& = \phi_{K_{i-1}} \cdots \phi_{K_1}(H) \\
& = K_i.
\end{align*}
Now since $H \cong \Z/p^k \times \Z/p^j$ where $k \leq j$ and $H_1 \cong \Z/p^{k-1} \times \Z/p^j$, we must have that $\phi_{L_s}^{-1}(K_i) = \QZ{2}[p]$. To see this, note that $[p^{k-1}]\phi_{L_{s-1}}\cdots \phi_{L_1}(H) \cong \Z/p^2$ would imply that $\QZ{2}[p^{k}] \nsubseteq H$, contradicting our assumption. It follows that
\[
K_i = \phi_{L_s}(\QZ{2}[p]) = L_s.
\] 
So $\phi_{K_i} = \phi_{L_s}$ and 
\begin{align*}
\phi_{K_i}\cdots \phi_{K_1} & = \phi_{K_i}[p^{k-1}]\phi_{L_s}\cdots\phi_{L_1} \\
& = [p^{k-1}]\phi_{L_s}^2\phi_{L_{s-1}}\cdots \phi_{L_1} \\
& =  [p^{k}]\phi_{L_{s-1}}\cdots \phi_{L_1}.
\end{align*}
This finishes the inductive construction. In order to see uniqueness, let $\phi_{K_i} \cdots \phi_{K_1}$ and $\phi_{K_{i}'} \cdots \phi_{K_{1}'}$ be two decompositions corresponding to two composition series for $H$ as above. Assume $\QZ{2}[p^k] \subseteq H$ and $\QZ{2}[p^{k+1}] \nsubseteq H$, then
\[
\phi_{K_i} \cdots \phi_{K_1} =  [p^k]\phi_{L_s} \cdots \phi_{L_1}
\]
and
\[
\phi_{K_{i}'} \cdots \phi_{K_{1}'} =  [p^k]\phi_{L_s'} \cdots \phi_{L_1'}. 
\]
We have that $[p^k]H \cong H/[p^k]H$ is cyclic and thus admits a unique decomposition. Therefore, $L_{\beta} = L_{\beta}'$ for $1\le \beta\le s$. 
\end{proof}

With this in hand, we can produce power sections for $n=2$.

\begin{prop} \label{combinatorics}
For each choice of endomorphisms for subgroups of order $p$ in $\QZ{2}$ satisfying the conclusion of \cref{orderp}, there is a power section $\phi \in \Gamma(\Sub(\qz), \Isog(\qz))$, i.e., a section $\phi$ such that for $K \subset H$
\[
\phi_H = \phi_{H/K}\phi_K, 
\]
where $H/K = \phi_K(H)$.
\end{prop}
\begin{proof}
With notation as in the previous lemma, define $\phi_H$ to be the composite $\phi_{K_i} \cdots \phi_{K_1}$. The power section condition for $\phi$ follows from the uniqueness of its normal form, given by \Cref{lem:powersectionh2}.
%
\end{proof}

\section{Recollections on Morava $E$-theory}\label{sec:etheoryrecolletions}
In this section we recall the facts that we need about Morava $E$-theory before we can proceed. For any space $X$ we set
\[
\E(X) = \E^0(X).
\]
We will also set $\E = \E^0$. By Goerss, Hopkins, Miller~\cite{structuredmoravae} there is a unique $E_{\infty}$-ring structure on the spectrum $\E$. In the homotopy category, this gives rise to total power operations \cite[Section VIII]{bmms}. The total power operations are natural multiplicative non-additive maps
\[
\Po_m \colon \E(X) \lra{} \E(E\Sigma_{m} \times_{\Sigma_{m}} X^{m}),
\]
for all $m > 0$, that satisfy the relations of \cref{porelations}. 

A common simplification of $\Po_m$ is given by restriction along the diagonal $X \lra{\Delta} X^{m}$. This gives a map
\[
P_m\colon\E(X) \lra{} \E(B\Sigma_{m} \times X) \cong \E(B\Sigma_{m}) \otimes_{\E} \E(X),
\]
the isomorphism being a consequence of the freeness of $\E(B\Sigma_{m})$ over $\E$ \cite[Proposition 3.6]{etheorysym}. Let $J \subset \E(E\Sigma_{p^k} \times_{\Sigma_{p^k}} X^{p^k})$ be the ideal generated by the image of the transfers along the maps
\[
E(\Sigma_i \times \Sigma_j) \times_{(\Sigma_i \times \Sigma_j)} X^{p^k} \lra{} E\Sigma_{p^k} \times_{\Sigma_{p^k}} X^{p^k}
\]
for $i,j > 0$ and $i+j = p^k$ and let $I \subset \E(B\Sigma_{p^k})$ be the ideal generated by the image of the same transfer for $X = \ast$. Neither $\Po_m$ nor $P_m$ are additive, but they can both be made so by taking the quotient by $J$ and $I$, respectively. Thus we have ring maps
\[
\Po_{p^k}/J\colon\E(X) \lra{} \E(E\Sigma_{p^k} \times_{\Sigma_{p^k}} X^{p^k})/J
\] 
and
\[
P_{p^k}/I\colon\E(X) \lra{} \E(B\Sigma_{p^k} \times X) \cong \E(B\Sigma_{p^k})/I \otimes_{\E} \E(X).
\]
Let us call $\Po_{p^k}/J$ the additive total power operation. For $m \neq p^k$ it is still possible to define $\Po_m/J$, but these ring maps are uninteresting. 

\begin{remark}
We have abused notation by calling these ideals $I$ and $J$ in conflict with the $I$ and $J$ defined prior to \cref{transferideal}. 
\end{remark}

We will apply these power operations to $X = BG$; in this case 
\[
E\Sigma_m \times_{\Sigma_m} X^m \simeq BG \wr \Sigma_m.
\]
Recall the maps of Equation \eqref{maps} from \cref{sec:rationalpo}. The relations that we will need are the following:

\begin{prop} (\cite{bmms}) \label{porelations}
For any $x \in \E(X)$ and $i,j,m\ge 0$, we have:
\begin{enumerate}
\item $\nabla^* \Po_{ij}(x) = \Po_i \Po_j(x)$,
\item $\Delta^* \Po_m(x) = x^m$,
\item $\Delta_{i,j}^* \Po_{i+j}(x) = \delta^*(\Po_i(x)\Po_j(x))$.
\end{enumerate}
\end{prop}
\begin{proof}
The proof can be found in Chapter~\upperRomannumeral{8} from \cite{bmms}.
\end{proof}

\begin{cor}[Proposition~\upperRomannumeral{8}.1.1, \cite{bmms}.] \label{Grewrite}
There is a commutative diagram
\[
\xymatrix{\E(BG) \ar[r]^-{\Po_{p^k}} \ar[d]^-{\Po_{p^{k-1}}} & \E(BG \wr \Sigma_{p^k}) \ar[d]^{\nabla^*} & \\ \E(BG \wr \Sigma_{p^{k-1}}) \ar[r]^-{\Po_{p}} & \E(BG \wr \Sigma_{p^{k-1}} \wr \Sigma_p) \ar[r]& \E(BG \wr \Sigma_{p^{k-1}}^p)}
\]
and the formula for the whole composite is just $x \mapsto \Po_{p^{k-1}}(x) \otimes \cdots \otimes \Po_{p^{k-1}}(x)$.
\end{cor}
\begin{proof}
The result follows immediately from \Cref{porelations} parts (1) and (3). 
\end{proof}


Character theory for Morava $E$-theory is constructed in \cite{hkr}. Hopkins, Kuhn, and Ravenel construct the ring $C_0$, which is introduced in \Cref{sec:isogenies}, and produce a natural character map
\[
\chi\colon\E(BG) \lra{} Cl_n(G,C_0)
\]
with the property that the induced map
\[
C_0 \otimes_{\E} \E(BG) \lra{\cong} Cl_n(G,C_0)
\]
is an isomorphism. They produce the action of $\Aut(\qz)$ on $Cl_n(G,C_0)$, described in \Cref{action}, and prove further that the above isomorphism induces an isomorphism
\[
p^{-1}\E(BG) \lra{\cong} Cl_n(G,C_0)^{\Aut(\qz)}.
\]
Note that $\E(BG)$ is a $\Z_p$-algebra, so $p^{-1}\E(BG) = \Q \otimes \E(BG)$ is the rationalization of the ring.

Theorem D in \cite{hkr} discusses the relationship between the character map and transfer maps for $\E$ and class functions. For $H \subset G$ there is a commutative diagram
\[
\xymatrix{\E(BH) \ar[r]^{\Tr} \ar[d]_{\chi} & \E(BG) \ar[d]^-{\chi} \\ Cl_n(H,C_0) \ar[r]^-{\Tr} & Cl_n(G,C_0),}
\] 
where 
\[
\Tr \colon Cl_n(H,C_0) \rightarrow Cl_n(G,C_0)
\]
is the transfer introduced in \cref{sec:rationalpo}.

There is a close relationship between the $E$-cohomology of finite groups and algebraic geometry related to $\G$. We will recall some of these results, and refer the interested reader to \cite{hkr, etheorysym,genstrickland} for the details. 

Let $\Hom(A^*, \G)$ be the scheme of maps from $A^*$ to $\G$ and define $\Sub_{p^k}(\G)$ to be the scheme of subgroups of order $p^k$ of $\G$, sending an $\E$-algebra $R$ to the collection of subgroups $H \subseteq R \otimes \G$ of order $p^k$. Moreover, for $A^* \subseteq \qz'$ a finite abelian group, $\Sub_{p^k|A|}^{A^*}(\G \oplus \qz')$ denotes the scheme with underlying functor
\[R \mapsto \{H \subseteq R\otimes \G \oplus \qz'\mid |H| = p^k|A|,\, \mathrm{pr}(H)=A^* \}\]
for an $\E$-algebra $R$, where $\mathrm{pr}$ is induced by the natural projection $\G \oplus \qz' \lra{} \qz'$. We often write $p^{k'} = p^k|A|$. The following table provides the dictionary we need. 

\begin{center}
\begin{tabular}{|l|l|}
\hline
Topology & Algebraic geometry \\
\hline
$\E(BA)$ & $\Hom(A^*, \G)$ \\
$\E(B\Sigma_{p^k})/I$ & $\Sub_{p^k}(\G)$ \\
$\E(BA \wr \Sigma_{p^k})/J$ & $\Sub_{p^{k'}}^{A^*}(\G \oplus \qz')$ \\
\hline
\end{tabular}
\end{center}

These results can be found in \cite[Proposition 5.12]{hkr}, \cite[Theorem 9.2]{etheorysym}, and \cite[Theorem 7.11]{genstrickland}.

\section{The additive total power operation applied to abelian groups}\label{sec:proofabelian}
Work of Ando~\cite{isogenies} as well as Ando, Hopkins, and Strickland~\cite{ahs} gives an algebro-geometric description of the additive power operation $P_{p^k}/I$ applied to finite abelian groups. In this section we recall their result and prove an extension, which is a key step in the proof of the main theorem.

\begin{lemma}\label{compatibletransferideal}
For any finite group $G$, there is an induced map
\[
\E(BG \wr \Sigma_{p^k})/J_{} \lra{} \E(BG)\otimes_{\E} \E(B\Sigma_{p^k})/I_{}.
\]
\end{lemma}
\begin{proof}
The claim reduces to checking that the following commutative diagram is a homotopy pullback for any $i,j$ with $i+j=p^k$:
\[\xymatrix{B(G \times \Sigma_i \times \Sigma_j) \ar[r] \ar[d] & BG\wr (\Sigma_i \times \Sigma_j) \ar[d]\\
B(G \times \Sigma_{p^k}) \ar[r] & BG \wr \Sigma_{p^k}.}\]
Since the square is commutative, it suffices to show that $BG \times \Sigma_i \times \Sigma_j$ has the correct homotopy type. This is obvious for $\pi_k$, $k>1$. To see that the pullback is connected, consider the double coset formula. Every element $g = (g_1,\ldots,g_{p^k};\sigma) \in G \wr \Sigma_{p^k}$ can be factored as $(g_1,\ldots,g_{p^k};e) \circ (1,\ldots,1;\sigma)$, where $(1,\ldots,1;\sigma) \in G\times \Sigma_{p^k}$ and $(g_1,\ldots,g_{p^k};e) \in G \wr (\Sigma_i \times \Sigma_j)$. Finally, $\pi_1$ of the pullback is given by the intersection of $G\times \Sigma_{p^k}$ with $G \wr (\Sigma_i \times \Sigma_j)$ inside $G \wr \Sigma_{p^k}$, which clearly is $G \times \Sigma_i \times \Sigma_j$.
\end{proof}

There are two $\E$-algebra structures on $\E(B\Sigma_{p^k})/I_{}$ of interest to us, given by the following two maps: the standard inclusion $i$ induced by $B\Sigma_{p^k} \lra{} \ast$ and the power operation $P_{p^k}/I\colon \E \to \E(B\Sigma_{p^k})/I$. Using \Cref{compatibletransferideal} and the commutative diagram
\[
\xymatrix{\E \ar[r]^-{P_{p^k}/I} \ar[d] & \E(B\Sigma_{p^k})/I \ar[d] \\ \E(BA) \ar[r]^-{\Po_{p^k}/J} & \E(BA\wr \Sigma_{p^k})/J,}
\]
we thus obtain a commutative diagram of rings:
\begin{equation}\label{relativeisogeny}
\xymatrix{\E(BA) \ar[r]_-{\Po_{p^k}/J} \ar[rd] \ar@/^1pc/[rr]^-{P_{p^k}/I} & \E(BA \wr \Sigma_{p^k})/J_{} \ar[r] & \E(BA) \otimes_{\E}^i \E(B\Sigma_{p^k})/I_{}\\
& \E(BA) \otimes_{\E}^{P_{p^k}/I} \E(B\Sigma_{p^k})/I_{}, \ar[u] \ar[ru]}
\end{equation}
where the superscripts on the tensor product indicate the relevant $\E$-algebra structure on $\E(B\Sigma_{p^k})/I$.

\begin{remark}
In our notation, $A$ plays a very different role than in the setting of \cite{ahs}. Indeed, they consider a level structure in place of our $\Sigma_{p^k}$ and $S^1$ instead of our $A$; the translation is readily made by passing to torsion subgroups of $S^1$ and using arbitrary subgroups, see \cite[Remark 3.12]{ahs}.
\end{remark}

The following result is proven in \cite[4.2.5]{isogenies} and \cite[3.21]{ahs}.

\begin{prop}[Ando--Hopkins--Strickland]\label{ahs}
The additive power operation
\[\E(BA) \lra{P_{p^k}/I} \E(BA) \otimes^{i}_{\E} \E(B\Sigma_{p^k})/I_{}\]
is the ring of functions on the map
\[\hom(A^*,\G) \times \Sub_{p^k}(\G) \lra{} \hom(A^*, \G)\]
induced by
\[(f\colon A^* \rightarrow R \otimes \G, H \subset R\otimes \G) \mapsto (A^* \lra{f} R\otimes \G \lra{} (R \otimes \G)/H \cong R \tensor[^{P_{p^k}/I}]{\otimes}{_\E} \G).\]
\end{prop}

Fix a map $\Z_p^t=\lat' \to A$ with $A$ finite abelian, and let $\qz' = (\lat')^*$. Let $p^{k'} = p^k|A|$. Recall from \cite{genstrickland} that there is a formal scheme $\Sub_{p^{k'}}^{A^*}(\G \oplus \qz')$  that associates to any $\E$-algebra $R$ the collection of subgroup schemes $H \subset R \otimes (\G \oplus \qz')$ of order $p^{k'}$ which project onto $A^*$ via the natural map $\G \oplus \qz' \lra{} \qz'$. The main result of \cite{genstrickland} implies that the ring of functions on $\Sub_{p^{k'}}^{A^*}(\G \oplus \qz')$ is isomorphic to $\E(BA \wr \Sigma_{p^k})/J_{}$. Our extension of \Cref{ahs} can now be stated as follows. 

\begin{thm}\label{abeliantotalpowerop}
The additive total power operation modulo the transfer
\[\E(BA) \lra{\Po_{p^k}/J} \E(BA \wr \Sigma_{p^k})/J\]
is the ring of functions on the map
\[\Q_{p^k}^*\colon \Sub_{p^{k'}}^{A^*}(\G\oplus \qz') \lra{} \hom(A^*, \G)\]
given by
\[(H \subset R\otimes(\G \oplus \qz')) \mapsto (A^* \lra{} R\otimes \G/K \cong R \tensor[^{P_{p^k}/I}]{\otimes}{_\E} \G),\]
where $K$ is the kernel in the map of short exact sequences
\[\xymatrix{K \ar[r] \ar[d]_{=} & H \ar[r] \ar[d] & A^* \ar[d] \\ K \ar[r] & \G \ar[r] & \G/K}\]
and $H$ maps to $\G$ through the projection $\G \oplus \qz' \twoheadrightarrow \G$.
\end{thm}

The content of the theorem is that two maps between $\E(BA)$ and $\E(BA \wr \Sigma_{p^k})/J$ are in fact the same map. The map $\Q_{p^k}$ is defined algebro-geometrically and the map $\Po_{p^k}/J$ is the additive total power operation. By \Cref{ahs}, we know that the maps are equal after mapping further to $\E(BA) \otimes_{\E} \E(B\Sigma_{p^k})/I$. But the map 
\[
\E(BA \wr \Sigma_{p^k})/J \lra{} \E(BA) \otimes_{\E} \E(B\Sigma_{p^k})/I 
\]
is not injective in general. The proof will proceed by building a ring that $\E(BA \wr \Sigma_{p^k})/J$ injects into, and that can be  attacked using \Cref{ahs}.

Before giving the proof, we draw a consequence that is of interest in its own right.
\begin{cor}
The map 
\[
\Po_{p^k}/J \otimes_{\E}^{P_{p^k}/I} \E(B\Sigma_{p^k})/I \colon \E(BA) \otimes_{\E}^{P_{p^k}/I} \E(B\Sigma_{p^k})/I_{} \lra{} \E(BA \wr \Sigma_{p^k})/J_{}
\]
given in Diagram \eqref{relativeisogeny} is an isomorphism. 
\end{cor}
\begin{proof}
This is the map that occurs in \Cref{abeliantotalpowerop} base changed to $\E(B\Sigma_{p^k})/I$. By \Cref{abeliantotalpowerop}, there is a commutative diagram
\[
\xymatrix{\E(BA) \ar[r]^-{\Po_{p^k}/J} \ar[d]_{\cong} & \E(BA \wr \Sigma_{p^k})/J_{} \ar[d]^{\cong} \\ \Gamma \hom(A^*, \G) \ar[r]^-{\Q_{p^k}} & \Gamma \Sub_{p^{k'}}^{A^*}(\G\oplus \qz'),}
\]
where the vertical isomorphisms are canonical. Thus it is enough to show that
\[
\Q_{p^k} \otimes_{\E}^{P_{p^k}/I} \E(B\Sigma_{p^k})/I_{}
\]
is an isomorphism. Proposition 6.5 of \cite{genstrickland} implies that it is an isomorphism.
\end{proof}

The idea of the proof of \Cref{abeliantotalpowerop} is to reduce the claim to the result of Ando--Hopkins--Strickland by probing $\E(BA \wr \Sigma_{p^k})/J$ by an appropriate family $\Fam{A \wr \Sigma_{p^k}}$ of abelian subgroups of $A \wr \Sigma_{p^k}$ that captures all of the transitive conjugacy classes. For each of the abelian subgroups $M \in \Fam{A \wr \Sigma_{p^k}}$, we show how the composite
\[
\E(BA) \lra{\Po_{p^k}/J} \E(BA \wr \Sigma_{p^k})/J \lra{} \E(BM)/I_M
\]
can be attacked using \Cref{ahs}, where $I_M$ is the ideal in $\E(BM)$ generated by the image of the transfer along all proper subgroups of $M$. This ring has been studied in \cite{ahs} and is closely related to $M$-level structures on $\G$.

\begin{definition}
For each $[\al] \in \hom(\lat,A\wr\Sigma_{p^k})^{\trans}_{/\sim}$, choose a representative that satisfies \Cref{factorizationlemma}. Let $\Fam{A \wr \Sigma_{p^k}}$ be the set of images of these representatives. 
\end{definition} 
By definition these subgroups fit into the commutative diagram
\begin{equation}\label{subgroupdiagram}
\xymatrix{K \ar[r]^{\iota} \ar[d]_{=} & M \ar[d] \ar[rd]^{\pi \al} \\
K \ar[r] \ar[d] & K \wr \Sigma_{p^k} \ar[r] \ar[d] & \Sigma_{p^k} \ar[d]^{=} \\
A \ar[r] & A \wr \Sigma_{p^k} \ar[r] & \Sigma_{p^k},}
\end{equation}
where $K$ denotes the pullback or, equivalently, the kernel of $\pi \alpha$ and $M \in \Fam{A \wr \Sigma_{p^k}}$.

\begin{lemma}\label{genstricklandlemma}
For any abelian group $A$, the map 
\[
\E(BA \wr \Sigma_{p^k})/J_{} \lra{} \Prod{{M \in \Fam{A \wr \Sigma_{p^k}}}}\E(BM)/I_{M}
\]
induced by restriction is injective.
\end{lemma}
\begin{proof}
We first need to show that the map exists. The map exists if the homotopy pullback of the diagram
\[
\xymatrix{ & B(A \wr (\Sigma_i \times \Sigma_j)) \ar[d] \\ BM \ar[r] & B(A \wr \Sigma_{p^k})} 
\]
is a disjoint union of classifying spaces of proper subgroups of $M$ when $i,j>1$ and $i+j = p^k$. The classifying spaces in the homotopy pullback are of the form 
\[
B(gMg^{-1} \cap (A \wr (\Sigma_i \times \Sigma_j)),
\]
where $g$ is a representative of a double coset in 
\[
M\backslash (A \wr \Sigma_{p^k})/(A \wr (\Sigma_i \times \Sigma_j)).
\]
The conjugate of a transitive subgroup is transitive and $A \wr (\Sigma_i \times \Sigma_j)$ is not transitive, thus the intersection cannot be all of $M$.

Since $\E(BA \wr \Sigma_{p^k})/J$ is a finitely generated free $\E$-module (Proposition 5.3 in \cite{genstrickland}), we may check injectivity after applying the character map. This allows us to check the claim on class functions, where it follows immediately from the construction of $\Fam{A \wr \Sigma_{p^k}}$.
\end{proof}

We now consider two maps between $M$ and $M\wr \Sigma_{p^k}$. The first is the composite
\[
M \hookrightarrow K \wr \Sigma_{p^k} \hookrightarrow M \wr \Sigma_{p^k},
\]
and the second is the composite
\[
\xymatrix{
M \ar[r]^-{id \times \pi \al} & M \times \Sigma_{p^k} \ar[r]^-{\Delta} & M \wr \Sigma_{p^k}.}
\]
Applying $E$-cohomology to these maps gives the same map because they are conjugate:

\begin{lemma}\label{liftingcriterionlemma}
The square 
\[\xymatrix{\E(BM \wr \Sigma_{p^k})/J_{} \ar[r] \ar[d] & \E(BM) \otimes_{\E}^i \E(B\Sigma_{p^k})/I_{} \ar[d] \\
\E(BK \wr \Sigma_{p^k})/J_{} \ar[r] & \E(BM)/I_{M}}\]
commutes, where the maps are induced by the maps described just above.
\end{lemma}
\begin{proof}
We show that the two maps $M \rightrightarrows M \wr \Sigma_{p^k}$ are conjugate. The composites 
\[
M \rightrightarrows M \wr \Sigma_{p^k} \rightarrow \Sigma_{p^k}
\]
coincide by construction. Thus, by \cref{subs}, it suffices to show that the kernel $K$ factors through $K \lra{\Delta} K\wr \Sigma_{p^k}$ by the identity map in each case. But this follows from the commutativity of the following two diagrams:
\[\xymatrix{K \ar[r]^{\iota} \ar[dd]_{\iota} & M \ar[ddr] \ar[d] & & K \ar[r] \ar[d]^{=} \ar@/_1pc/[dd]_{\iota} & M \ar[ddr] \ar[d] \\
& M \times \Sigma_{p^k} \ar[d] & & K \ar[r] \ar[d] & K \wr \Sigma_{p^k} \ar[rd] \ar[d] \\
M \ar[r]^-{\Delta} & M \wr \Sigma_{p^k} \ar[r] & \Sigma_{p^k} & M \ar[r]^-{\Delta} & M \wr \Sigma_{p^k} \ar[r] & \Sigma_{p^k}.}\]
Thus we have a commutative diagram in $E$-cohomology without taking any quotients. For each ring we may take the quotient by the appropriate transfer ideal to get the square in the statement of the lemma.
\end{proof}

We are now ready to put the pieces together to prove the main theorem of this section. 

\begin{proof}[Proof of \Cref{abeliantotalpowerop}]
By \Cref{genstricklandlemma}, we may reduce to showing that the following two composites coincide
\[
\xymatrix{\E(BA) \ar@<0.5ex>[r]^-{\Po_{p^k}/J} \ar@<-0.5ex>[r]_-{\Q_{p^k}} & \E(BA \wr \Sigma_{p^k})/J_{} \text{ } \ar@{^{(}->}[r] & \Prod{M \in \Fam{A \wr \Sigma_{p^k}}}\E(BM)/I_{M}.}
\]
We will prove this one factor at a time. Let $M \in \Fam{A \wr \Sigma_{p^k}}$ and consider the diagram
\[\xymatrix{\E(BM) \ar@<0.5ex>[r]^-{\Po_{p^k}/J} \ar@<-0.5ex>[r]_-{\Q_{p^k}} \ar@{->>}[d]_f & \E(BM \wr \Sigma_{p^k})/J_{} \ar[r] \ar[d] & \E(BM) \otimes_{\E}^i \E(B\Sigma_{p^k})/I_{} \ar[d] \\
\E(BK) \ar@<0.5ex>[r]^-{\Po_{p^k}/J} \ar@<-0.5ex>[r]_-{\Q_{p^k}} & \E(BK \wr \Sigma_{p^k})/J_{} \ar[r] & \E(BM)/I_{M} \\
\E(BA) \ar[u] \ar@<0.5ex>[r]^-{\Po_{p^k}/J} \ar@<-0.5ex>[r]_-{\Q_{p^k}} & \E(BA\wr \Sigma_{p^k})/J. \ar[u] \ar[ur]}\]
The diagram commutes, where the $\Q_{p^k}$'s commute with $\Q_{p^k}$'s and the $\Po_{p^k}/J$'s commute with $\Po_{p^k}/J$'s; for the top right square this follows from \Cref{liftingcriterionlemma}. The composites on the top row are equal by \Cref{ahs}. The surjectivity of $f$ implies that the middle two composites are equal. But now that these are equal, the composites from $\E(BA)$ to $\E(BM)/I_M$ must be equal. This proves the claim.
\end{proof}

\section{The character of the total power operation}\label{sec:proofgeneral}
In this section we establish the relevance of the multiplicative natural transformations constructed in \Cref{sec:rationaltrafo}. The goal of this section is to prove the following theorem. 
\begin{thm}\label{clmainthm}
Let $\Po_{m} \colon \E(BG) \lra{} \E(BG\wr \Sigma_{m})$ be the total power operation for Morava $E$-cohomology applied to $BG$, and let $\chi\colon\E(BG) \lra{} Cl_n(G,C_0)$ be the character map. For all $n \geq 0$, all $m \geq 0$, and any section $\phi \in \Gamma(\Sub(\qz), \Isog(\qz))$, there is a commutative diagram
\[
\xymatrix{\E(BG) \ar[r]^-{\Po_{m}} \ar[d]_{\chi} & \E(BG \wr \Sigma_{m}) \ar[d]^{\chi} \\ Cl_n(G,C_0) \ar[r]^-{\Po_{m}^{\phi}} & Cl_n(G\wr \Sigma_{m},C_0),}
\]
which is natural in $G$.
\end{thm}

We will prove the theorem in three steps. First, our extension of the work of Ando, Hopkins, and Strickland in \Cref{abeliantotalpowerop} can be used to prove the theorem for the additive total power operation applied to finite abelian groups. To extend this to all finite groups, we use a modification of the fact that $\E(BG)$ rationally embeds under the restriction map into $\prod_{A \subseteq G} \E(BA)$, where the product is over all abelian subgroups of $G$. Finally, an inductive argument using character theory extends the result from the additive total power operation to the total power operation.

\begin{prop}\label{pomodtransferab}
For a finite abelian group $A$ the diagram
\begin{equation}\label{eq:abaddpo}
\xymatrix{\E(BA) \ar[r]^-{\Po_{p^k}/J} \ar[d] & \E(BA\wr\Sigma_{p^k})/J \ar[d] \\ Cl_n(A,C_0) \ar[r]^-{\Po_{p^k}^{\phi}/J} & Cl_n(A \wr \Sigma_{p^k},C_0)/J}
\end{equation}
commutes.
\end{prop}
\begin{proof}
By \cref{transferideal}, the terminal object in the square is the product
\[
Cl_n(A \wr \Sigma_{p^k},C_0)/J \cong \Prod{\Sub_{p^k}(\qz,A)} C_0.
\]
Thus it suffices to fix an element $(H, [\al]) \in \Sub_{p^k}(\qz,A)$ and prove the result for the factor corresponding to $(H,[\al])$. Since $\Po_{p^k}^{\phi}/J$ is the rationalization of a map between products of the Drinfeld ring $\D$ by \cref{drinfeldringproducts}, it suffices to replace $C_0$ with $\D$. For $R$ a complete local ring, a map $\D \lra{} R$ out of the factor corresponding to $(H, [\al])$ is the data
\begin{equation} \label{data}
(H \subset \qz, \al\colon\lat_H \lra{} A, \qz \hookrightarrow R \otimes_{\E} \G).
\end{equation}
We are suppressing the data of the Lubin-Tate moduli problem. Note that, since $A$ is abelian, $\al = [\al]$ and that we may replace $\al$ by
\[
\al^*\colon A^* \lra{} \qz/H.
\]
Let $\lat' = \Z_{p}^t$ and $\qz' = (\lat')^*$, where $t$ is greater than or equal to the number of generators of $A$. Let $\lat' \twoheadrightarrow A$ be a surjection. By the algebro-geometric description of $\E(BA \wr \Sigma_{p^k})/J$ in Theorem 7.11 of \cite{genstrickland}, the right vertical map in \eqref{eq:abaddpo} sends the data in \eqref{data} to the pullback
\[
\xymatrix{B \ar[r] \ar[d] & \G \oplus \qz' \ar[d] \\ A^* \ar[r] & (R \otimes_{\E} \G)/H \oplus \qz'.}
\]
The top horizontal map in \eqref{eq:abaddpo} sends this data to the composite 
\begin{equation}\label{eq:theothermap}
A^* \lra{} (R \otimes_{\E} \G)/H \cong R \tensor[^{Q_H}]{\otimes}{_{\E}} \G
\end{equation}
by \Cref{ahs}.

Going around the other way first sends the data to the pair of composites
\[
(A^* \lra{\al^*} \qz/H \lra{\psi_H} \qz, \qz \lra{\psi_{H}^{-1}} \qz/H \lra{} R \tensor[^{Q_H}]{\otimes}{_{\E}} \G)
\]
and then composes them to give
\[
A^* \lra{} R \tensor[^{Q_H}]{\otimes}{_{\E}} \G,
\]
which is the same as the map in \eqref{eq:theothermap}.
\end{proof}

\begin{lemma} \label{embedding}
There is an embedding
\[
p^{-1}\E(BG \wr \Sigma_{p^k})/J \hookrightarrow \Prod{A \subseteq G}p^{-1}\E(BA \wr \Sigma_{p^k})/J.
\]
\end{lemma}
\begin{proof}
Since $C_0$ is a faithfully flat $p^{-1}\E$-algebra it suffices to check this on class functions. Thus the claim is equivalent to there being a surjection of sets
\[
\Coprod{A \subseteq G} \Sub_{p^k}(\qz,A) \lra{} \Sub_{p^k}(\qz,G).
\]
This is a surjection because any map $\lat_H \lra{} G$ factors through its image, which is an abelian subgroup of $G$.
\end{proof}

\begin{prop}\label{pomodtransfer}
For $G$ a finite group, the diagram
\[
\xymatrix{\E(BG) \ar[r]^-{\Po_{p^k}/J} \ar[d] & \E(BG\wr\Sigma_{p^k})/J \ar[d] \\ Cl_n(G,C_0) \ar[r]^-{\Po_{p^k}^{\phi}/J} & Cl_n(G \wr \Sigma_{p^k},C_0)/J}
\]
commutes.
\end{prop}
\begin{proof}
The map from the top arrow to the bottom arrow factors through the rationalization because $\Po_{p^k}/J$ is a ring map and $C_0$ is a rational algebra. It suffices to consider the following cube
\[\resizebox{\columnwidth}{!}{
\xymatrix{& p^{-1}\E(BG) \ar[ld]_{\Po_{p^k}/J} \ar'[d][dd] \ar[rr] & & \Prod{A \subseteq G} p^{-1}\E(BA) \ar[ld]^{\prod \Po_{p^k}/J} \ar[dd] \\
p^{-1}\E(BG \wr \Sigma_{p^k})/J \ar[dd] \ar[rr] & &  \Prod{A\subseteq G}p^{-1}\E(BA\wr\Sigma_{p^k})/J \ar[dd] \\
& Cl_n(G,C_0) \ar[ld]_{\Po_{p^k}^{\phi}/J} \ar'[r][rr] & &  \Prod{A \subseteq G}Cl_n(A,C_0) \ar[ld]^{\prod \Po_{p^k}^{\phi}/J} \\
Cl_n(G\wr \Sigma_{p^k})/J \ar[rr]  &&  \Prod{A\subseteq G}Cl_n(A\wr \Sigma_{p^k})/J.}
}\]
The top and bottom squares commute by naturality, the back and front squares commute by character theory (Theorems C and D in \cite{hkr}), and the right square commutes by \Cref{pomodtransferab}. Now, since the horizontal maps are injections, the left square must commute.
\end{proof}

\begin{prop}\label{injectioncl}
For all $k > 0$ and any finite group $G$, there is an injection
\[
Cl_n(G \wr \Sigma_{p^k},C_0) \hookrightarrow (Cl_n(G \wr \Sigma_{p^k},C_0)/J) \times Cl_n(G \wr \Sigma_{p^{k-1}}^{p},C_0),
\]
where the map to the left factor is the quotient and the map to the right factor is given by restriction.
\end{prop}
\begin{proof}
It is just a matter of checking this on conjugacy classes. Every element of $\Sub_{p^k}(\qz,G) \allowbreak \subseteq \Sum_{p^k}(\qz,G)$ is hit by the left factor. By \cref{sums}, the map to the other factor is induced by the map 
\[
\prod_{l=1}^{p}\Sum_{p^{k-1}}(\qz,G) \twoheadrightarrow \Sum_{p^k}(\qz,G) \setminus \Sub_{p^k}(\qz,G)
\]
defined by
\[
\prod_{l = 1}^{p} \, \bigoplus_{i = 1}^{j_l} (H_{i,l}, [\al_{i,l}]) \mapsto \bigoplus_{l=1}^p\bigoplus_{i = 1}^{j_l}(H_{i,l}, [\al_{i,l}])
\]
which is clearly surjective. Note that it is not an isomorphism as it sends an ordered collection of sums to their (unordered) sum.
\end{proof}

The following proposition is the base case of an induction on $k$ in the proof of \cref{clmainthm}.

\begin{prop}\label{inductionstart}
There is a commutative diagram
\[
\xymatrix{\E(BG) \ar[r]^-{\Po_p} \ar[d]_{\chi} & \E(BG \wr \Sigma_{p}) \ar[d]^{\chi} \\ Cl_n(G,C_0) \ar[r]^-{\Po_{p}^{\phi}} & Cl_n(G\wr \Sigma_{p},C_0).}
\]
\end{prop}
\begin{proof}
Consider the following diagram
\begin{equation} \label{diagram}
\xymatrix{\E(BG) \ar[r]^-{\Po_{p}} \ar[d] & \E(BG \wr \Sigma_{p}) \ar[d] \ar[r] & \E(BG^{p}) \times (\E(BG \wr \Sigma_p)/J) \ar[d] \\ Cl_n(G,C_0) \ar[r]^-{\Po_{p}^{\phi}} & Cl_n(G\wr \Sigma_{p},C_0) \ar[r] & Cl_n(G^{p},C_0) \times (Cl_n(G \wr \Sigma_p,C_0)/J).}
\end{equation}
The right vertical arrow is just the product of character maps. First we show that the outer rectangle 
\begin{equation}\label{eq:square9.6}
\xymatrix{\E(BG) \ar[r] \ar[d] & \E(BG^{p}) \times (\E(BG \wr \Sigma_p)/J) \ar[d] \\ Cl_n(G,C_0) \ar[r] & Cl_n(G^{p},C_0) \times (Cl_n(G \wr \Sigma_p,C_0)/J)}
\end{equation}
commutes. It commutes for the right factor by \Cref{pomodtransfer}. 

The commutativity of the left factor is proven as follows. It is a result of \cite{hkr} that the image of $\E(BG)$ is in the $\Aut(\qz) \cong GL_n(\Z_p)$ invariants of $Cl_n(G,C_0)$, where the $\Aut(\qz)$-action is the action induced by \Cref{action}. Recall that for $f \in Cl_n(G, C_0)$, the action takes the form
\[
f^{\phi_e}([\al]) = \phi_{e}^* f([\al \psi_{e}^*]),
\]
where $\psi_{e} = \phi_{e}$ since it factors the identity map:
\[
\xymatrix{\qz \ar[dr]_{\id = q_e} \ar[r]^{\phi_e} & \qz  \\ & \qz/e = \qz. \ar[u]_{\psi_{e}}}
\]
Next note that the inclusion $G^p \hookrightarrow G \wr \Sigma_p$ induces the map on conjugacy classes sending
\[
([\al_i])_{i = 1\ldots p} \mapsto \Oplus{i = 1\ldots p} (e,[\al_i]).
\]
Thus the composite of the bottom arrows sends
\[
f \mapsto \otimes_i f^{\phi_e}([\al_i]).
\]
For $x \in \E(BG)$, $\chi(x)$ is fixed by this action, so $\chi(x)$ maps to $\otimes_i \chi(x) = \chi(\otimes_i x)$. Now by \cref{Grewrite} for $k=1$, the square in \eqref{eq:square9.6} commutes.

Finally, the right square in \eqref{diagram} commutes by naturality of the character map. Since the bottom right arrow is an injection by \Cref{injectioncl}, 
the left square must commute as well.
\end{proof}

Now we finish the induction.

\begin{prop} \label{tpofinite}
The following diagram commmutes:
\[
\xymatrix{\E(BG) \ar[r]^-{\Po_{p^k}} \ar[d]_{\chi} & \E(BG \wr \Sigma_{p^k}) \ar[d]^{\chi} \\ Cl_n(G,C_0) \ar[r]^-{\Po^{\phi}_{p^k}} & Cl_n(G\wr \Sigma_{p^k},C_0).}
\]
\end{prop}
\begin{proof}
Consider the following diagram
\[\resizebox{\columnwidth}{!}{
\xymatrix{\E(BG) \ar[r]^-{\Po_{p^k}} \ar[d] & \E(BG \wr \Sigma_{p^k}) \ar[d] \ar[r] & \E(BG\wr \Sigma_{p^{k-1}}^p) \times (\E(BG)\otimes_{\E} \E(\Sigma_{p^k})/I_{}) \ar[d] \\ Cl_n(G,C_0) \ar[r]^-{\Po^{\phi}_{p^k}} & Cl_n(G\wr \Sigma_{p^k},C_0) \ar[r] & Cl_n(G \wr \Sigma_{p^{k-1}}^p,C_0) \times (Cl_n(G,C_0) \otimes_{C_0} Cl_n(\Sigma_{p^k},C_0)/I_{}).}
}\]
The outer rectangle commutes because we understand the maps to each factor in the right lower corner. The map to the left factor is determined by induction and \Cref{Grewrite}. The map to the right factor commutes by \Cref{pomodtransfer}. The bottom right map is an injection by \Cref{injectioncl}. Thus the left square must commute.
\end{proof}

Finally, we finish the proof of \Cref{clmainthm}.

\begin{proof}[Proof of \Cref{clmainthm}]
Let $\sum_j a_j p^j$ be the $p$-adic expansion of $m$. Then the inclusion of groups
\[
\Prod{j} \Sigma_{p^j}^{a_j} \lra{} \Sigma_m
\]
induces a commutative square
\[
\xymatrix{\E(BG\wr \Sigma_m) \ar[r] \ar[d] & \E(\Prod{j}BG \wr \Sigma_{p^j}^{a_j}) \ar[d] \\ Cl_n(G \wr \Sigma_m,C_0) \ar@{^{(}->}[r] & Cl_n(\Prod{j} G\wr \Sigma_{p^j}^{a_j},C_0),} 
\]
in which the bottom arrow is an injection by \Cref{padicsum}. By \Cref{porelations} part (3), the composite
\[
\E(BG) \lra{\Po_m} \E(BG \wr \Sigma_m) \lra{} \E(\Prod{j}BG \wr \Sigma_{p^j}^{a_j})
\]
is the external tensor product $\otimes_j \Po_{p^j}^{\otimes a_j}$. Consider the following diagram:
\[
\xymatrix{\E(BG) \ar[r] \ar[d] & \E(BG\wr \Sigma_m) \ar[r] \ar[d] & \E(\Prod{j}BG \wr \Sigma_{p^j}^{a_j}) \ar[d] \\Cl_n(G,C_0) \ar[r] & Cl_n(G \wr \Sigma_m,C_0) \ar@{^{(}->}[r] & Cl_n(\Prod{j} G\wr \Sigma_{p^j}^{a_j},C_0).}
\]
By \Cref{tpofinite}, we know that the outer rectangle commutes, and the right square commutes by naturality of the character map. This implies that the left square commutes.
\end{proof}

\begin{example}
In Proposition 3.6.1 of \cite{isogenies}, Ando constructs Adams operations for Morava $E$-theory. Reformulating his construction in terms of the power operation for $\E$ shows that the Adams operations are the composite
\[
\psi^{p^k}\colon \E(BG) \lra{P_{p^{kn}}/I} \E(BG) \otimes_{\E} \E(B\Sigma_{p^{kn}})/I \lra{} \E(BG),
\]
where the last map is induced by the map
\[
\E(B\Sigma_{p^{kn}})/I \lra{} \E,
\]
picking out the subgroup $\G[p^k] \subset \G$. \Cref{adamsops} computes the same composite on class functions. As a special case of \Cref{clmainthm}, we have a commutative diagram
\[
\xymatrix{\E(BG) \ar[r]^-{\psi^{p^k}} \ar[d]_-{\chi} & \E(BG) \ar[d]^-{\chi} \\ Cl_n(G,C_0) \ar[r]^{\psi^{p^k}} & Cl_n(G,C_0).}
\]
When $\kappa \subset \F_{p^n}$ and $\G$ is the universal deformation of the Honda formal group, this gives a generalization of the well-known formula from representation theory stating that, for a representation $\rho$,
\[
\chi(\psi^m(\rho))(g) = \chi(\rho)(g^m).
\]
\end{example}

\section{The rational total power operation}\label{sec:maintheorem}
Recall that there is an action of $\Aut(\qz)$ on $Cl_n(G, C_0)$. It is Theorem C of \cite{hkr} that there is a canonical isomorphism
\[
p^{-1} \E(BG) \cong Cl_n(G, C_0)^{\Aut(\qz)}.
\]
In this section we prove that, for any section $\phi$, $\Po^{\phi}_{m}$ sends $\Aut(\qz)$-invariants to $\Aut(\qz)$-invariants and that the restriction of $\Po^{\phi}_{m}$ to the $\Aut(\qz)$-invariants is independent of the choice of $\phi$. The resulting ``rational total power operation" is a global power functor. 

\begin{thm} \label{rationalized}
For all finite groups $G$ and any section $\phi \in \Gamma(\Sub(\qz), \Isog(\qz))$, the function
\[
\Po^{\phi}_{m} \colon Cl_n(G,C_0) \lra{} Cl_n(G\wr \Sigma_m,C_0)
\]
sends $\Aut(\qz)$-invariants to $\Aut(\qz)$-invariants. By restricting $\Po_{m}^{\phi}$ to the $\Aut(\qz)$-invariants, this gives rise to a commutative diagram
\[\xymatrix{\E(BG) \ar[r]^-{\mathbb{P}_{m}} \ar[d] & \E(BG \wr \Sigma_{m}) \ar[d] \\
p^{-1}\E(BG) \ar[r]_-{\mathbb{P}^{\phi}_{m}} & p^{-1}\E(BG \wr \Sigma_{m}).}\]
\end{thm}
\begin{proof}
By Theorem C of \cite{hkr}  and \Cref{clmainthm}, it is thus enough to show that the multiplicative natural transformation 
\[\Po^{\phi}_{m}\colon Cl_n(G,C_0) \lra{} Cl_n(G\wr \Sigma_{m},C_0)\] 
restricts to the $\Aut(\qz)$-fixed points of both sides. To this end, recall from \Cref{sumsaction} that the action of $\sigma \in \Aut(\qz)$ on $\bigoplus_i(H_i,[\alpha_i]) \in \Sum_m(\qz,G)$ is given by
\begin{equation}\label{eq:glaction}
\sigma\colon \bigoplus_i(H_i,[\alpha_i]) \mapsto \bigoplus_i(\sigma H_i,[\alpha_i \sigma^{*}_{|_{\lat_{H_i}}}]),
\end{equation}
with notation as in the following commutative diagram:
\[
\xymatrix{\lat_{\sigma H_i} \ar[r] \ar[d]_{\sigma^*_{|_{\lat_{\sigma H_i}}}} & \lat \ar[d]^{\sigma^*} \\
\lat_{H_i} \ar[r] & \lat.}
\]

Fix an automorphism $\sigma \in \Aut(\qz)$ and an invariant element $f \in Cl_n(G,C_0)^{\Aut(\qz)}$. Since $\phi_{H_i}$ and $\phi_{\sigma H_i}\sigma$ have the same kernel, there exists a unique isomorphism $\gamma_i \in \Aut(\qz)$ making the diagram
\[\xymatrix{\qz \ar[r]^{\sigma} \ar[d]_{\phi_{H_i}} & \qz \ar[d]^{\phi_{\sigma H_i}} \\
\qz \ar[r]_{\gamma_i}^{\cong} & \qz}\]
commute. Upon dualizing and using the identities $\phi_{H_i} = \psi_{H_i} \circ q_{H_i}$ and $\phi_{\sigma H_i} = \psi_{\sigma H_i} \circ q_{\sigma H_i}$, we see that the inner squares and triangles in the next diagram commute:
\[\xymatrix{\lat_{H_i} \ar[rd]^{q_{H_{i}}^*} & & & \lat_{\sigma H_{i}} \ar[lll]_{\sigma^*_{|_{\lat_{\sigma H_i}}}} \ar[ld]_-{q^*_{\sigma H_{i}}} \\
& \lat & \lat \ar[l]_-{\sigma^*} \\
& \lat \ar[u]_{\phi_{H_{i}}^*} \ar[luu]^{\psi_{H_{i}}^*} & \lat \ar[l]^{\gamma^*} \ar[u]^{\phi_{\sigma H_{i}}^*} \ar[ruu]_{\psi_{\sigma H_{i}}^*}.}\]
Therefore, the outer diagram commutes as well and hence gives 
\begin{equation}\label{eq:gammatrafo}
\sigma^{*}_{|_{\lat_{\sigma H_i}}} \psi_{\sigma H_i}^* = \psi_{H_i}^*\gamma^*.
\end{equation}
Now we can check that $\Po^{\phi}_{m}(f)$ is invariant under the action of $\Aut(\qz)$:
\begin{align*}
\mathbb{P}^{\phi}_{m}(f)^{\sigma}(\bigoplus_i(H_i,[\alpha_i])) & = \sigma^*\Prod{i}\mathbb{P}^{\phi}_{m}(f)((\sigma H_i,[\alpha_i \sigma^*_{|_{\lat_{\sigma H_i}}}])) & & \text{by \eqref{eq:glaction}}  \\
& = \Prod{i}\sigma^*\phi_{\sigma H_i}^*f([\alpha_i\sigma^*_{|_{\lat_{\sigma H_i}}} \psi_{\sigma H_i}^*]) \\
& = \Prod{i}\phi_{H_i}^*\gamma_{i}^*f([\alpha_i\psi_{H_i}^*\gamma_{i}^*]) & & \text{by \eqref{eq:gammatrafo}} \\
& = \Prod{i}\phi_{H_i}^*f([\alpha_i\psi_{H_i}^*]) & & \text{as $f \in Cl_n(G,C_0)^{\Aut(\qz)}$}\\
& = \Prod{i}f^{\phi_{H_i}}([\alpha_i]) \\
& = \mathbb{P}^{\phi}_{m}(f)(\bigoplus_i(H_i,[\alpha_i])).
\end{align*}
\end{proof}

\begin{cor} \label{independent}
The restriction of $\Po^{\phi}_{m}$ to the $\Aut(\qz)$-invariants
\[\Po^{\phi}_{m}\colon p^{-1}\E(BG) \lra{} p^{-1}\E(BG \wr \Sigma_{m})\]
is independent of the chosen section $\phi$.
\end{cor}
\begin{proof}[First proof]
As in the proof of \Cref{clmainthm}, we can reduce to the case $m=p^k$. By naturality and \Cref{embedding}, it furthermore suffices to prove this for abelian groups, since $p^{-1}\E(BG)$ embeds into the product of the rationalized $E$-cohomology of the abelian subgroups of $G$. We have two maps
\[
p^{-1}\E(BA) \lra{} p^{-1}\E(BA \wr \Sigma_{p^k})/J;
\]
the first is the rationalization $\Q\otimes \Po_{p^k}/J$ of $\Po_{p^k}/J$ and the second is $\Po^{\phi}_{p^k}/J$ restricted to the $\Aut(\qz)$-fixed points. Because $\E(BA)$ is a finitely generated free $\E$-module, we may choose a basis of $\E(BA)$ which thus gives a basis for $p^{-1}\E(BA)$. By \Cref{rationalized} both maps send the basis elements to the same elements of the codomain, thus the maps are the same. 

We get the full result by induction. We use the embedding 
\[
\xymatrix{p^{-1}\E(BA \wr \Sigma_{p^k}) \ar@{^{(}->}[r] & p^{-1}\E(BA \wr \Sigma_{p^k})/J \times p^{-1}\E(BA \wr \Sigma_{p^{k-1}}^{p})} 
\]  
and induct on $k$. The base case is clear and the induction follows by considering the large diagram and right diagram just as in \Cref{tpofinite}.
\end{proof}

\begin{proof}[Second proof]
We give a second proof of the corollary which is intrinsic to the construction of $\Po^{\phi}_{m}$ and in particular does not rely on properties of Morava $E$-theory. To this end, consider two sections $\phi,\phi' \in \Gamma(\Sub(\qz), \Isog(\qz))$ with associated isomorphisms $\psi_H,\psi'_H\colon \qz/H \lra{\cong} \qz$ for $H \subset \qz$ as in \Cref{sec:notation}. For a fixed $\bigoplus_i(H_i,[\alpha_i]) \in \Sum_m(\qz,G)$, take $\gamma_{H_i} \in \Aut(\qz)$ to be the unique automorphism making the following diagram commute
\begin{equation}\label{eq:gammaauto}
\xymatrix{\qz \ar[r]^{=} \ar[d]_{\phi_{H_i}'} & \qz \ar[d]^{\phi_{H_i}} \\
\qz \ar[r]_{\gamma_{H_i}}^{\cong} & \qz.}
\end{equation}
As in the proof of \Cref{rationalized} with $\sigma = \id$, we see that, for any $f \in Cl_n(G,C_0)^{\Aut(\qz)}$, 
\begin{align*}
\mathbb{P}^{\phi}_{m}(f)(\bigoplus_i(H_i,[\alpha_i])) & = \Prod{i}\phi_{H_i}^*f([\alpha_i\psi_{H_i}^*]) \\
& = \Prod{i}(\phi_{H_i}')^*\gamma_{H_i}^*f([\alpha_i(\psi_{H_i}')^*\gamma_{H_i}^*]) & & \text{by \eqref{eq:gammaauto}} \\
& = \Prod{i}(\phi_{H_i}')^*f([\alpha_i(\psi_{H_i}')^*]) & & \text{as $f \in Cl_n(G,C_0)^{\Aut(\qz)}$}  \\
& = \mathbb{P}^{\phi'}_{m}(f)(\bigoplus_i(H_i,[\alpha_i])),
\end{align*}
hence $\mathbb{P}^{\phi}_{m} = \mathbb{P}^{\phi'}_{m}$ on $p^{-1}\E(BG) \cong Cl_n(G,C_0)^{\Aut(\qz)}$.
\end{proof}
It follows that $\Q \otimes \Po_{p^k}/J = \Po^{\phi}_{p^k}/J$ after restricting $\Po^{\phi}_{m}/J$ to $p^{-1}\E(BG)$. Therefore, the following definition makes sense.

\begin{definition}
For any section $\phi$, let 
\[
\xymatrix{\Po^{\Q}_m\colon p^{-1}\E(BG) \ar[r] & p^{-1}\E(BG \wr \Sigma_m)}
\]
be the restriction of $\Po_{m}^{\phi}$ to $p^{-1}\E(BG)$. We will call this the rational total power operation.
\end{definition}

\begin{thm}\label{ratpoglobalpower}
The rational total power operation $\Po^{\Q}_m$ is a global power functor.
\end{thm}
\begin{proof}
We must show that the diagram
\[
\xymatrix{p^{-1}\E(BG) \ar[r]^-{\Po_{m}^{\Q}} \ar[d]_{\Po_{s}^{\Q}} & p^{-1}\E(BG \wr \Sigma_{m}) \ar[d]^{\nabla} \\ p^{-1}\E(BG \wr \Sigma_{s}) \ar[r]_-{\Po_{t}^{\Q}} & p^{-1}\E(BG \wr \Sigma_{s} \wr \Sigma_{t})}  
\]
commutes, where $st = m$ and $\nabla$ is induced by the natural inclusion.

This is a computation that follows the lines of the proof of \cref{powersecthm}. In \cref{powersecthm} we work with a power section and here we work with an $\Aut(\qz)$-invariant class functions. We will use the notation of \cref{powersecthm}. 

Choose a section $\phi \in \Gamma(\Sub(\qz), \Isog(\qz))$. Since $\phi_{T_{i,j}}$ and $\phi_{T_{i,j}/H_i} \phi_{H_i}$ have the same kernel, there exists an automorphism $\sigma \in \Aut(\qz)$ such that 
\[
\phi_{T_{i,j}} = \sigma \phi_{T_{i,j}/H_i} \phi_{H_i}.
\]
Diagram \ref{bigdiagram} implies that 
\[
\psi_{T_{i,j}}^{*} = \psi_{H_i}^*|_{\lat_{\psi_{H_i}(K_{i,j})}} \psi_{T_{i,j}/H_i}^* \sigma^*.
\]

Now assume that $f \in Cl_n(G,C_0)^{\Aut(\qz)}$. The global power structure of $\phi$ in the proof of \cref{powersecthm} is used in going from line 7 to line 8 in the sequence of equalities. Up to line 7 we have the same sequence of equalities that give
\[
\Po^{\phi}_{t}(\Po^{\phi}_{s}(f))(\bigoplus_i(H_i, \bigoplus_j (K_{i,j},[(\al_{H_i})_{K_{i,j}}]))) = \Prod{i}\Prod{j}\phi_{H_i}^*\phi_{T_{i,j}/H_i}^*f([(\al_{H_i})_{K_{i,j}} \psi_{H_i}^*|_{\lat_{\psi_{H_i}(K_{i,j})}}\psi_{T_{i,j}/H_i}^*]).
\]
Now we use the fact that $f$ is $\Aut(\qz)$-invariant to get
\begin{align*}
\Prod{i}\Prod{j}\phi_{H_i}^*\phi_{T_{i,j}/H_i}^*f&([(\al_{H_i})_{K_{i,j}} \psi_{H_i}^*|_{\lat_{\psi_{H_i}(K_{i,j})}}\psi_{T_{i,j}/H_i}^*]) \\
&=\Prod{i}\Prod{j}\phi_{H_i}^*\phi_{T_{i,j}/H_i}^*\sigma^*f([(\al_{H_i})_{K_{i,j}} \psi_{H_i}^*|_{\lat_{\psi_{H_i}(K_{i,j})}}\psi_{T_{i,j}/H_i}^*\sigma^*])\\
&=\Prod{i}\Prod{j}\phi_{T_{i,j}}^*f([(\al_{H_i})_{K_{i,j}} \psi_{T_{i,j}}^*]) \\
&=\Po^{\phi}_{m}(f)(\bigoplus_{i,j}(T_{i,j},[(\al_{H_i})_{K_{i,j}}])). \\
\end{align*}

\end{proof}

In particular, this generalizes a result of Ganter~\cite[Proposition 4.12]{globalgreen} to arbitrary heights and answers a question that she poses after the proof of Proposition 4.12 in \cite{globalgreen}. 

\begin{cor}\label{charglobalpower}
The character map 
\[\chi\colon \E(B-) \lra{} p^{-1}\E(B-) \cong Cl_n(-,C_0)^{\Aut(\qz)}\]
is a map of global power functors. 
\end{cor}

\begin{example}
We give an example of how these theorems may be used. \Cref{ahs} describes Ando, Hopkins, and Strickland's algebro-geometric interpretation of the additive power operation
\[
P_{p^k}/I\colon \E \lra{} \E(B\Sigma_{p^k})/I.
\]
Let 
\[
L \subset \E(B\Sigma_{p^k}\wr \Sigma_{p^h}) 
\]
be the ideal generated by the image of the transfer along the maps
\[
(\Sigma_i \times \Sigma_j) \wr \Sigma_{p^h} \lra{} \Sigma_{p^k}\wr \Sigma_{p^h},
\]
where $i,j>0$ and $i+j = p^k$, as well as the maps
\[
\Sigma_{p^k} \wr (\Sigma_i \times \Sigma_j)  \lra{} \Sigma_{p^k}\wr \Sigma_{p^h},
\]
where $i,j>0$ and $i+j = p^h$. It is a folklore result of Rezk (now proved by Nelson in \cite{nelson}) that 
\[
\E(B\Sigma_{p^k}\wr \Sigma_{p^h})/L
\]
is finitely generated and free as an $\E$-module and corepresents the scheme $\Sub_{p^k,p^h}(\G)$ of flags of subgroup schemes $H_0 \subseteq H_1 \subset \G$, where $|H_0| = p^k$ and $|H_1/H_0| = p^h$.

By considering the power operation $\Po_{p^h}$ applied to the transfer along $\Sigma_i \times \Sigma_j \subset \Sigma_{p^k}$ for $i,j>0$ and $i+j = p^k$, there is a commutative ring map
\[
\E(B\Sigma_{p^k})/I \lra{P_{p^h}/L} \E(B\Sigma_{p^k}\wr \Sigma_{p^h})/L.
\]
It is also natural to consider the composite
\[
\E \lra{P_{p^k}/I} \E(B\Sigma_{p^k})/I \lra{P_{p^h}/L} \E(B\Sigma_{p^k}\wr \Sigma_{p^h})/L.
\]
Both $P_{p^h}/L$ and $P_{p^h}/L \circ P_{p^k}/I$ can be understood algebro-geometrically by using \cref{clmainthm}.

There is a natural map of formal schemes from flags of subgroups to subgroups
\[
Z \colon \Sub_{p^k,p^h}(\G) \rightarrow \Sub_{p^h}(\G)
\]
given by
\[
(H_0 \subseteq H_1 \subset \G) \mapsto (H_1/H_0 \subset \G/H_0)
\]
and a map
\[
z \colon \Sub_{p^k,p^h}(\G) \rightarrow \Sub_{p^h}(\G) \rightarrow \Spf(\E) 
\]
given by
\[
(H_0 \subseteq H_1 \subset \G) \mapsto \G/H_1.
\]
\cref{clmainthm} gives a way to see that the algebro-geometric maps and the power operations agree. 
Using \Cref{ratpoglobalpower}, it is easy to check that both maps make the following diagram commute
\[
\xymatrix{\E \ar@<0.5ex>[r]^-{P_{p^h}/L \circ P_{p^k}/I} \ar@<-0.5ex>[r]_-{z^*} \ar[d] & \E(B\Sigma_{p^k}\wr \Sigma_{p^h})/L \ar[d] \\ C_0 \ar[r] & Cl_n(B\Sigma_{p^k}\wr \Sigma_{p^h},C_0)/L.} 
\]
Since the vertical maps are injective, this implies that both $P_{p^h}/L \circ P_{p^k}/I$ and $z^*$ are the same map.
\end{example}

%
%
%
\bibliographystyle{amsalpha}
\bibliography{mybib}



\end{document}